\documentclass[a4paper, 11pt, reqno]{amsart}

\usepackage[T1]{fontenc}
\usepackage[utf8]{inputenc}
\usepackage[english]{babel}

\usepackage{tikz}
\usepackage{tikz-cd}
\usetikzlibrary{arrows,decorations.pathreplacing,calligraphy,patterns,positioning,matrix,calc,bending,decorations.pathmorphing,tikzmark}
\usepackage{amsfonts}
\usepackage{amssymb}
\usepackage{amsmath}
\usepackage{enumitem}
\usepackage{bm}

\newcommand*{\faktorDisplay}[2]{%
  \raisebox{0.5\height}{\ensuremath{#1}}
  \mkern-5mu\diagup\mkern-4mu
  \raisebox{-0.5\height}{\ensuremath{#2}}
}
\newcommand{\faktor}[2]{#1/#2}
\usepackage{float}
\usepackage{graphicx}
\usepackage{wrapfig}
\usepackage{subcaption}
\definecolor{USred}{cmyk}{0,1.00,0.65,0.34}
\definecolor{USblue}{cmyk}{1.00,0.65,0,0.34}
\definecolor{USyellow}{cmyk}{0,0.14,0.81,0.1}
\definecolor{USbrown}{cmyk}{0,0.45,0.77,0.29}
\definecolor{USgreen}{cmyk}{0.84,0,1,0.55}

\usepackage{amsthm}

\theoremstyle{definition}
\newtheorem{de}{Definition}[section]

\theoremstyle{plain}
\newtheorem{thm}[de]{Theorem}
\newtheorem{mainthm}{Theorem}

\newtheorem*{thm*}{Theorem}

\newtheorem{lem}[de]{Lemma}
\newtheorem{p}[de]{Proposition}
\newtheorem{cor}[de]{Corollary}

\theoremstyle{definition}

\theoremstyle{plain}

\usepackage{Sections/Commands}
\newcommand\Stretch[1]{\scalebox{.7}[1]{#1}}
\RequirePackage{doi}
\usepackage{hyperref}
\usepackage[libertine]{newtxmath}
\usepackage[ttscale=.875]{libertine}

\title{Principal minors of the distance matrix of a tree}
\author[\'A.~Guti\'errez]{Álvaro Gutiérrez}
\address{\textsuperscript{1} School of Mathematics, University of Bristol, United Kingdom.}
\email{a.gutierrezcaceres@bristol.ac.uk}
\author[A.~Lillo]{Adrián Lillo}
\address{\textsuperscript{2} Departamento de Álgebra, Facultad de Matemáticas, Universidad de Sevilla, Spain.}
\email{alillo@us.es}
\date{1st July, 2024}
\begin{document}

\begin{abstract}
Let $T = ([n], E)$ be a tree and let $D = ( d(i,j) )_{i, j \le n}$ be the distance matrix of $T$. Let $S\subseteq [n]$. We give the first combinatorial proof for a formula to compute the principal minor of $D$ indexed by $S$, namely $\det D[S]$. This generalizes work of Graham and Pollak, as well as more recent works.
\end{abstract}

\maketitle
\begin{center}
    \itshape
As presented in the 30th British Combinatorial Conference
\end{center}

\setcounter{secnumdepth}{1}
\setcounter{tocdepth}{1}

\section{Introduction}
Consider a tree $T = ([n], E)$ with $n$ vertices. The distance between two vertices $i$ and $j$, denoted by $d(i,j)$, is defined as the number of edges in the unique path $P(i,j)$ from $i$ to $j$ in $T$. Let $D = \big(d(i,j)\big)_{1\le i,j\le n}$ be the distance matrix of $T$. Graham and Pollak \cite{GrahamPollak} gave in 1971 the following formula for the determinant of the distance matrix of a tree:
\[
\det D = (-1)^{n-1}(n-1)2^{n-2}.
\]
Of particular significance, this formula only depends on the number of vertices of $T$, and not on its structure.

There exist several elementary proofs of Graham and Pollak's formula \cite{GrahamPollak, GrahamLovasz:1978:Orsay, YanYeh:2007, TilliaThesis,  ZhouDing, DuYeh}, as well as many generalizations to weighted trees
\cite{BapatKirklandNeumann, BapatLalPati2009, ZhouDing} and $q$-analogues of the above \cite{BapatLalPati, YanYeh:2007, LSZ}.
So far, all these generalized notions of distance are \emph{additive}, in the sense that the distance from $i$ to $j$ can be expressed as the sum of certain polynomials depending on the intermediate paths $P(k,j)$ as $k$ ranges over the nodes of $P(i,j)$. Multiplicative generalizations of distances were explored in \cite{YanYeh,ZDmult}. 
Recently, a very general framework was develop by Choudhury and Khare in \cite{CK19}, which gives a formula that specializes to all of the above. (Also for general graphs, see \cite{CK23}.) Moreover, the formula also specializes to compute \emph{some} principal and non-principal minors of the matrix. A similar but different expression to compute \emph{all} principal minors of the matrix was very recently given by Richman, Shokrieh, and Wu \cite{richmanminors}. A multiplicative distance analogue is \cite{Hirai}.

The first involutive proof of Graham and Pollak's formula was given by Briand, Esquivias, Rosas and the authors of the present paper in \cite{FPSAC, BEGLR}. Furthermore, the determinant of all of the additive distance matrices generalized above follow easily from the bijections. The main idea is to reduce the problem to a path enumeration problem, and to invoke the Lindström--Gessel--Viennot Lemma \cite{Lindstrom, Gessel-Viennot} to conclude.

In the present paper, we give a new formula for the principal minors of the distance matrix of a tree.
Our proof is bijective and based on \cite{BEGLR}. Our methods can be used to deduce formulas given in \cite{CK19, richmanminors}. 

\subsection{Outline of paper}
We begin by fixing graph theoretic notation and stating our main results in Section \ref{sec:statement}.
We turn the problem into an enumeration of combinatorial objects coined catalysts in Section \ref{sec: main combinatorial objects}, after which we are ready to give an outline of the proof of Theorem \ref{thm:main result}. The details of the proof are postponed to Propositions \ref{prop:zero-sum}, \ref{prop:unital}, \ref{prop:composite}. We devote Section \ref{sec:zero-sum} to the proof of Proposition \ref{prop:zero-sum}. Sections \ref{sec:quotient} and \ref{sec:route maps} reduce the remainder of problem even further to a path enumeration problem. In order to do this, several combinatorial constructions have to be careful examined. Much of this is delegated to \cite{BEGLR} and \cite{derangements}, where the analysis takes up the bulk of the works.
We show Proposition \ref{prop:unital} in Section \ref{sec:unital} and Proposition \ref{prop:composite} in Section \ref{sec:composite}.
We conclude by deducing \cite{richmanminors} in Section \ref{sec: richman as coro}, and a Corollary of \cite{CK19} in Section \ref{sec: CK as coro} and giving some closing remarks.

\section{Statement of the main theorem}
\label{sec:statement}
If $X$ is a finite set, we denote as $\mathbb S_X$ the set of permutations of $X$. When $X = [k]$ for some integer $k \geq 0$, we simplify the notation to $\mathbb S_k$, as is customary.

Throughout the paper, fix a tree $T = ([n], E)$ with $n$ vertices, and consider a subset $S \subseteq [n]$ of $m\geq 2$ vertices. We let $E^\pm$ be the set of all possible oriented arcs supported on $T$. That is, for each $\{i,j\}\in E$, we have $(i,j)\in E^\pm$ and $(j,i)\in E^\pm$. 

Given a simple graph $G$, we let $\edge{G}$ be the number of edges of $G$ and $\cc{G}$ be the number of connected components of $G$. For instance, $\edge{T} = n-1$ and $\cc{T} = 1$.
A subgraph of a simple graph $G = (V,E)$ is a graph $H = (V',E')$ where $V'\subseteq V$ and $E'\subseteq E$.
We write $H\subseteq G$.
For subgraphs $H\subseteq T$ of a tree, we have $\cc{H} = n - \edge{H}$.
A \emph{spanning forest} of a graph $G$ is an acyclic subgraph containing every vertex of $G$. The union of two simple graphs $G = (V,E)$ and $G' = (V',E')$ is the simple graph $G\cup G' = (V\cup V', E\cup E')$.

We borrow from \cite{richmanminors} the following concepts: 
\begin{de}
\label{de: S-rooted}
    Let $T = ([n],E)$ be a tree, let $S\subseteq[n]$ be a subset of cardinality $m$.
    A spanning forest $F$ of $T$ is said to be \emph{$S$-rooted} if it has $m$ connected components and each of its connected components has exactly one vertex in $S$. See Figure \ref{fig: tree and forests C}.
    We write $F = \bigsqcup_{s\in S} F_s$, where $F_s$ is the component of $F$ containing $s$. 
    We denote the set of $S$-rooted spanning forests of $T$ by $\Srooted$.
    
    Similarly, a spanning forest $F$ is said to be \emph{($S,*$)-rooted} if it has $m + 1$ connected components, and exactly one of them does not contain any vertex of $S$. See Figure \ref{fig: tree and forests D}. We write $F = F_*\sqcup\bigsqcup_{s\in S} F_s$, where $F_s$ is the component of $F$ containing $s$.  We call $F_*$ the \emph{floating component of $F$}. The set of $(S, *)$-rooted spanning forests of $T$ is denoted  $\SSrooted$.
\end{de}

\begin{figure}[h]
    \centering\
        \begin{subfigure}{0.22\textwidth}
        \includegraphics[page = 1]{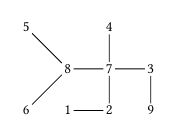}
        \subcaption{A tree $T$.}
        \label{fig: tree and forests A}
    \end{subfigure}\hfill
    \begin{subfigure}{0.22\textwidth}
        \includegraphics[page = 2]{Figures/TreeAndForest.pdf}
        \subcaption{A set $S$.}
        \label{fig: tree and forests B}
    \end{subfigure}\hfill
    \begin{subfigure}{0.22\textwidth}
        \includegraphics[page = 3]{Figures/TreeAndForest.pdf}
        \subcaption{$S$-rooted forest.}
        \label{fig: tree and forests C}
    \end{subfigure}\hfill
    \begin{subfigure}{0.22\textwidth}
        \includegraphics[page = 4]{Figures/TreeAndForest.pdf}
        \subcaption{$(S,*)$-rooted forest.}
        \label{fig: tree and forests D}
    \end{subfigure}
    \caption{}
    \label{fig: tree and forests}
\end{figure}

Let $G$ be a simple graph, let $F = \bigsqcup_{i \in I} F_i$ be a spanning forest of $G$, with $F_i = (V_i, E_i)$. The \emph{boundary of a component $F_{i_0}$}, denoted $\partial F_{i_0}$, is the set of nodes of $G$ adjacent to $F_{i_0}$ in $\bigsqcup_{i \in I, i\ne i_0} F_i$. That is,
\[
\partial F_{i_0} = \{u \in V - V_{i_0}\ : \ \exists \{u,v\} \in E,~ v\in V_{i_0}\}.
\]
The \emph{boundary degree} of $F_{i_0}$ is defined as the number of such nodes, 
\[
\bdeg{F_{i_0}} := \# \partial F_{i_0}.
\]
Similarly, given a vertex $x$ of the tree, we write $\partial x$ for the set of its neighbors and $\delta(x) = \# \partial x$ for its degree. (Since the discrete graph supported on $[n]$ is a spanning forest, the notation is consistent with the above.)

\begin{mainthm}
\label{thm:main result}
Let $T = ([n], E)$ be a tree and let $D$ be its distance matrix. Let $S \subseteq [n]$ be a subset of cardinality $m\ge2$. Then
    \[
\det D[S] = (-1)^{m-1}2^{m-2}\left((m - 1)\kappaTS - \!\!\!\!\sum_{F \in \SSrooted} \!\!\!\!\big(\bdeg{F_*} - 1\big)\big(\bdeg{F_*} - 4\big)\right).
\]
\end{mainthm}

Compare this with the formula given in \cite{richmanminors} (see Theorem \ref{thm:richman main result}). As a sanity check: if $S = [n]$ then $\#\Srooted = 1$ and $\#\SSrooted = 0$, and our formula specializes to the Graham--Pollak formula.

\section{\texorpdfstring{$S$}{S}-catalysts and \texorpdfstring{$m$}{m}-arrowflows}
\label{sec: main combinatorial objects}
\subsection{Catalysts}
\label{subsec: catalysts}
\begin{de}
    Given a permutation $\sigma$ in $\mathbb S_S$ and a map $f: S \to E^\pm$, we say that the ordered pair $(\sigma, f)$ is an \emph{$S$-catalyst for $T$} if, for each vertex $i$ of $S$, the arc $f(i)$ lies in the path from $i$ to $\sigma(i)$, and is oriented according to this path.
    (Note that the endpoints of $f(i)$ do not necessarily belong to $S$.)
    We denote by $K_S$ the set of all $S$-catalysts for $T$.
\end{de}
In particular, when $S = [n]$, an $S$-catalyst for $T$ is a catalyst for $T$ in the sense of \cite{BEGLR}. Catalysts were introduced in \cite{BEGLR} as the natural object for which $\det D$ is a signed enumeration. Indeed, since there are $d(i,\sigma(i))$ edges in the path from $i$ to $\sigma(i)$, then
\begin{equation}
\label{eqn:sum of catalysts}
  \det D[S] = \sum_{\sigma\in\mathbb S_S} \sgn(\sigma) ~ d(s_1,\sigma(s_1)) \cdots d(s_m,\sigma(s_m)) = \sum_{(\sigma, f) \in K_S} \sgn(\sigma).
\end{equation}
\begin{ej}
    Let $T$ and $S$ be the tree and set of Figure \ref{fig: tree and forests}. Figure \ref{fig: catalysts} showcases four $S$-catalysts for $T$. We can represent them as follows:
    \[\scriptsize
    \def\arraycolsep{0.2em}
    \renewcommand{\arraystretch}{.7}
    \left(\begin{array}{rccccc}
    & 1 & 3 & 4 & 5 & 6\\
    \sigma & 5 & 6 & 3 & 4 & 1\\
    f & 85 & 37 & 47 & 87 & 87
    \end{array}\right),
    \left(\begin{array}{rccccc}
    & 1 & 3 & 4 & 5 & 6\\
    \sigma & 3 & 5 & 6 & 4 & 1\\
    f & 27 & 37 & 47 & 87 & 21
    \end{array}\right),
    \left(\begin{array}{rccccc}
    & 1 & 3 & 4 & 5 & 6\\
    \sigma & 3 & 4 & 5 & 6 & 1\\
    f & 73 & 37 & 47 & 86 & 87
    \end{array}\right),
    \left(\begin{array}{rccccc}
    & 1 & 3 & 4 & 5 & 6\\
    \sigma & 3 & 4 & 1 & 6 & 5\\
    f & 73 & 74 & 21 & 58 & 68
    \end{array}\right).
    \]
    For instance, for the $S$-catalyst $(\sigma,f)$ of Figure \ref{fig: catalysts a}, $\sigma(1) = 5$ and $f(1) = (8,5)$. This is illustrated with a path from $1$ to $5$ with a mark at $(8,5)$.
\end{ej}
\begin{figure}[h]
    \centering
    \begin{subfigure}{0.22\textwidth}
        \includegraphics[page = 1]{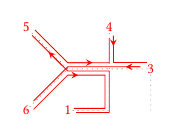}
        \subcaption{}\label{fig: catalysts a}
    \end{subfigure}\hfill
    \begin{subfigure}{0.22\textwidth}
        \includegraphics[page = 2]{Figures/catalysts.pdf}
        \subcaption{}\label{fig: catalysts b}
    \end{subfigure}\hfill
    \begin{subfigure}{0.22\textwidth}
        \includegraphics[page = 3]{Figures/catalysts.pdf}
        \subcaption{}\label{fig: catalysts c}
    \end{subfigure}\hfill
    \begin{subfigure}{0.22\textwidth}
        \includegraphics[page = 4]{Figures/catalysts.pdf}
        \subcaption{}\label{fig: catalysts d}
    \end{subfigure}
    \caption{Four $S$-catalysts for $T$.}
    \label{fig: catalysts}
\end{figure}

\subsection{Arrowflows}
\label{subsec: arrowflows}
\begin{de}
    Let $k$ be a nonnegative integer. We define a \emph{$k$-arrowflow on $T$} to be a directed multigraph with vertex set $[n]$, with $k$ arcs (counted with multiplicity), and whose underlying simple graph is a subgraph of $T$. 
\end{de}
\begin{figure}[h]
    \centering
    \begin{subfigure}{0.22\textwidth}
        \includegraphics[page = 1]{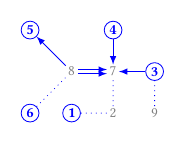}
        \subcaption{Parallel arrows.}\label{fig: arrowflows a}
    \end{subfigure}\hfill
    \begin{subfigure}{0.22\textwidth}
        \includegraphics[page = 2]{Figures/arrowflows.pdf}
        \subcaption{Missing path.}\label{fig: arrowflows b}
    \end{subfigure}\hfill
    \begin{subfigure}{0.22\textwidth}
        \includegraphics[page = 3]{Figures/arrowflows.pdf}
        \subcaption{Unital.}\label{fig: arrowflows c}
    \end{subfigure}\hfill
    \begin{subfigure}{0.22\textwidth}
        \includegraphics[page = 4]{Figures/arrowflows.pdf}
        \subcaption{Composite.}\label{fig: arrowflows d}
    \end{subfigure}
    \caption{Four $m$-arrowflows on $T$.}
    \label{fig: arrowflows}
\end{figure} 
If $k = n$, we emphasize this saying that the $k$-arrowflow is \emph{complete}. A complete $k$-arrowflow is an arrowflow in the sense of \cite{BEGLR}.

 Given an $S$-catalyst $\kappa = (\sigma, f)$ for $T$, we define the \emph{$m$-arrowflow induced by $\kappa$} as the $m$-arrowflow on $T$ with arc multiset $\{\!\{ f(i)\ |\ i\in S\}\!\}.$ 

\begin{ej}
    The four arrowflows of Figure \ref{fig: arrowflows} are induced from the four catalysts of Figure \ref{fig: catalysts}.
\end{ej}

The \emph{$m$-arrowflow class of $A$ for $S$}, denoted by $C_S(A)$, is defined as the set of $S$-catalysts for $T$ inducing the $m$-arrowflow $A$. The non-empty arrowflow classes form a partition of the set of $S$-catalysts, allowing us to rewrite Equation  \eqref{eqn:sum of catalysts} as
\[
  \det D[S] = \sum_{A} \sum_{\kappa \in C_S(A)}\sgn(\kappa)
\]
where $A$ ranges over the set of $m$-arrowflows on $T$.

To better understand this sum, we distinguish three different classes of $m$-arrowflows.
Given a digraph $A = (V,E)$, its \emph{underlying simple graph} is the graph $\SG(A)$ with vertex set $V$ and edge set $\{\{u,v\}\ : \ (u,v) \in E\}$. 
Given a simple graph $G = (V,E)$ and a subgraph  $H = (V,E')$ with same vertex set, the \emph{complement} of $H$ in $G$ defined as the subgraph $H^\comp := (V, E-E')$ of $G$. Note $G = H \cup H^\comp$.
Back in our setting, the \emph{missing forest of an arrowflow $A$} is the complement of the underlying simple graph of $A$, namely $\missing{A}$. Note that $\missing{A}$ is a spanning forest of $T$.
 
\begin{de}[Zero-sum arrowflow]
    Let $A$ be an $m$-arrowflow on $T$.
    \begin{itemize}
        \item If there is an arc $a\in E^{\pm}$ that appears with multiplicity at least $2$ in $A$, we say that $A$ has \emph{parallel arrows}.
        \item If there is a path $P(i,j)$ in $\missing{A}$ between a pair of vertices $i, j$ of $S$, we say $A$ has a \emph{missing path (for $S$)}.
    \end{itemize}
    An $m$-arrowflow is \emph{zero-sum (for $S$)} if it has either parallel arrows or a missing path for $S$.
\end{de}
\begin{de}[Unital arrowflow]
    An $m$-arrowflow $A$ is \emph{unital (for $S$)} if it has no parallel arrows and its missing forest is $S$-rooted, $\missing{A}\in\Srooted$.
\end{de}
\begin{de}[Composite arrowflow]
    An $m$-arrowflow $A$ is \emph{composite (for $S$)} if it has no parallel arrows and its missing forest is $(S,*)$-rooted, $\missing{A}\in\SSrooted$.
\end{de}

\begin{ej}
    The arrowflow of Figure \ref{fig: arrowflows a} has parallel arrows at $(8,7)$, and thus it is a zero-sum arrowflow. The arrowflow of Figure \ref{fig: arrowflows b} has a missing path $P(5,6)$. The remaining two arrowflows are unital and composite, respectively. Indeed, the missing forest of Figure \ref{fig: arrowflows c} is the $S$-rooted forest of Figure \ref{fig: tree and forests C}, and the missing forest of Figure \ref{fig: arrowflows d} is the $(S,*)$-rooted forest of Figure \ref{fig: tree and forests D}.
\end{ej}

\begin{lem}\label{lem: arrowflow partition}
    Every $m$-arrowflow is either zero-sum, unital, or composite for $S$.
\end{lem}
\begin{proof}
    Let $A$ be an $m$-arrowflow on $T$.
    The number of connected components of $\missing{A}$ is
    \begin{align*}
        \cc{\missing{A}} &= n - \edge{\missing{A}}\\
        &= n - \big((n-1) - \edge{\SG(A)}\big)\\
        &= 1 + \edge{\SG(A)}\\
        &\le 1 + \edge{A}\\
        &= m+1.   
    \end{align*}
    If $A$ has a missing path then $A$ is zero-sum. Suppose otherwise; then each vertex of $S$ lies in a different connected component of $\missing{A}$. In particular, we get $m\le \cc{\missing{A}} \le m+1$ and the missing forest $\missing{A}$ is either $S$-rooted or $(S,*)$-rooted. Altogether, $A$ is either zero-sum, unital, or composite.
\end{proof}

\section{Proof of the main theorem}
\label{subsec: proof of Thm A}
The proof of Theorem \ref{thm:main result} will follow from the following propositions, whose proofs will unfold over the next sections.
\begin{p}\label{prop:zero-sum}
Let $A$ be a zero-sum arrowflow. Then, 
\[
\sum_{\kappa\in C_S(A)} \sgn(\kappa) = 0.
\]
\end{p}
\begin{p}\label{prop:unital}
Let $A$ be a unital arrowflow. Then, 
\[
\sum_{\kappa\in C_S(A)} \sgn(\kappa) = (-1)^{m-1}.
\]
Let $F \in \Srooted$ be an $S$-rooted forest.
The number of unital arrowflows with missing forest $F$ is $(m-1)2^{m-2}$.
\end{p}
\begin{p}\label{prop:composite}
Let $F\in\SSrooted$ be an $(S,*)$-rooted spanning forest of $T$. Then,
\[
\sum_{\substack{A\\\text{composite}\\ \missing{A} = F}}\sum_{\kappa\in C_S(A)} \sgn(\kappa) =(-1)^{m} 2^{m - 2}(\bdeg{F_*} - 1)(\bdeg{F_*} - 4).
\]
\end{p}
We end this section with the proof of Theorem \ref{thm:main result}, which we recall below.
\begin{thm*}[A]
Let $T = ([n], E)$ be a tree and let $D$ be its distance matrix. Let $S \subseteq [n]$ be a subset of cardinality $m\ge2$. Then
    \[
\det D[S] = (-1)^{m-1}2^{m-2}\left((m - 1)\kappaTS - \!\!\!\!\sum_{\SSrooted} \!\!\!\!\big(\bdeg{F_*} - 1\big)\big(\bdeg{F_*} - 4\big)\right).
\]
\end{thm*}
\begin{proof}[Proof of Theorem \ref{thm:main result}]
    Equation \eqref{eqn:sum of catalysts} establishes
    \[
        \det D[S] = \sum_{(\sigma,f)\in K_S} \sgn(\sigma) = \sum_A \sum_{\kappa\in C_S(A)}\sgn(\kappa).
    \]
    We expand this sum into three summands and get
    \[
    \det D[S] = 
    \sum_{\substack{A\\\text{zero-sum}}} \sum_{\kappa\in C_S(A)}\sgn(\kappa) +
    \sum_{\substack{A\\\text{unital}}} \sum_{\kappa\in C_S(A)}\sgn(\kappa) + 
    \sum_{\substack{A\\\text{composite}}} \sum_{\kappa\in C_S(A)}\sgn(\kappa).
    \]
    The first summand vanishes by Proposition \ref{prop:zero-sum}. By Proposition \ref{prop:unital}, the second summand is
    \[
    (-1)^{m-1}(m-1)2^{m-2}\kappaTS.
    \]
    Finally, by Proposition \ref{prop:composite}, the last summand is
    \begin{align*}
        \sum_{\substack{A\\\text{composite}}} \sum_{\kappa\in C_S(A)}\sgn(\kappa)
        &= \sum_{F\in\SSrooted}\sum_{\substack{A\\\text{composite}\\\missing{A}=F}} \sum_{\kappa\in C_S(A)}\sgn(\kappa)\\
        &= \sum_{F\in\SSrooted} (-1)^{m}2^{m-2}(\bdeg{F_*}-1)(\bdeg{F_*}-4).
    \end{align*}
    Altogether, we obtain the desired formula.
\end{proof}

\section{Zero-sum arrowflows}
\label{sec:zero-sum}
This section is dedicated to a combinatorial proof of Proposition \ref{prop:zero-sum}. Given a zero-sum arrowflow $A$, we construct a sign-reversing involution on the set $C_S(A)$ of catalysts that induce $A$. We conclude that the signed enumeration of such catalysts is $0$. In order to do this, we adapt the  constructions of \cite{BEGLR}.
\begin{lem}\label{lem:zero-sum involution}
  Let $A$ be a zero-sum $m$-arrowflow on $T$. If $A$ has a missing path, let $i$ and $j$ be its endpoints. On the other hand, if $A$ has parallel arrows, let $i$ and $j$ be the two preimages of those arrows under $ f$. Then, the map 
  \begin{align*}
      \varphi: ~~ C_S(A) &\to C_S(A)\\
      (\sigma, f)&\mapsto(\sigma \circ (i\ j), f \circ (i\ j))
  \end{align*}
  is a sign-reversing involution.
\end{lem}
\begin{proof}
The map is clearly an involution if it is well defined. To see that it is well defined, we need to show that $ f(i)$ is an arc in both $P(j,\sigma(i))$ and $P(i, \sigma(i))$.

Suppose $A$ has parallel arrows, $ f(i) =  f(j)$. Then, $i$ and $j$ lie in the same connected component of $T-\{ f(i)\}$, whereas $\sigma(i)$ and $\sigma(j)$ lie in the other component. This shows the claim.

Suppose $A$ has a missing path $P(i,j) \subseteq \missing{A}$.
Now $P(j,\sigma(i))$ and $P(i,\sigma(i))$ coincide everywhere except in a subgraph of $P(i,j)$. Since $ f(i)$ is not in $P(i,j)$, it must appear with the same orientation in both paths. See Figure \ref{fig:empty_involution}.   \end{proof}

\begin{figure}[h]
  \centering
  \raisebox{-.5\totalheight}{\includegraphics[page = 1]{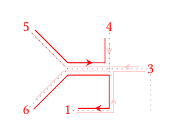}}
  \quad $\stackrel{\varphi}{\longleftrightarrow}$ \quad
  \raisebox{-.5\totalheight}{\includegraphics[page = 2]{Figures/zerosumInvolution.pdf}}
  \caption{Involution $\varphi$ on an $m$-arrowflow class $C_S(A)$, where $A$ is as in Figure \ref{fig: arrowflows b} and has missing path $P(5, 6)$.}
  \label{fig:empty_involution}
\end{figure}
\section{Quotient arrowflows}
\label{sec:quotient}
Given a simple graph $G = (V,E)$ and a partition $V = \bigsqcup_{i\in I} V_i$ of its vertex set, we define the \emph{quotient graph $G/\sim_I$ with respect to this partition} as the graph obtained from $G$ by contracting each $V_i$.
That is, $G/\sim_I$ has vertex set $I$ and an edge $\{i,j\}$ whenever $i\neq j$ and there exists $u \in V_i$ and $v \in V_j$ such that $\{u,v\}\in E$. Similarly, let $A$ be a digraph with underlying simple graph $\SG(A) = G$ as above. The quotient \emph{multigraph $A/\sim_I$ with respect to this partition} is defined as the multigraph obtained from $A$ by contracting each $V_i$.

\begin{de}[Quotient arrowflow]
    Let $A$ be an arrowflow, let $F = \missing{A} = \bigsqcup_{i\in I} F_i$ be its missing forest, $F_i = (V_i, E_i)$ for each $i\in I$. Note that $V = \bigsqcup_{i\in I} V_i$ is a partition of $V$.
    The \emph{quotient arrowflow $A/\sim$} is defined to be the quotient multigraph with respect to this partition. See Figure \ref{fig: quotient unital}.
\end{de}

\begin{figure}[h]
    \centering
    \begin{subfigure}{0.32\textwidth}
        \centering
        \includegraphics[page = 1]{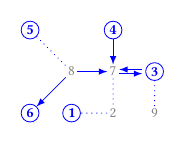}
        \subcaption{An arrowflow $A$.}
        \label{fig: quotient unital a}
    \end{subfigure}
    \begin{subfigure}{0.32\textwidth}
        \centering
        \includegraphics[page = 2]{Figures/quotientUnital.pdf}
        \subcaption{Its missing forest $\missing{A}$.}
        \label{fig: quotient unital b}
    \end{subfigure}
    \begin{subfigure}{0.32\textwidth}
        \centering
        \includegraphics[page = 3]{Figures/quotientUnital.pdf}
        \subcaption{The quotient arrowflow $A/\sim$.}
        \label{fig: quotient unital c}
    \end{subfigure}
    \caption{ }
    \label{fig: quotient unital}
\end{figure}
 In Lemma \ref{lem: quotient arc bijection}, we slightly abuse notation by identifying an arrowflow $A$ and its quotient $A/\sim$ with their respective sets of arcs (without multiplicity).
\begin{lem}
\label{lem: quotient arc bijection}
Let $A$ be an arrowflow with missing forest $\bigsqcup_{i\in I} F_i$, and let $\psi: [n] \to I$ denote the quotient projection.
Then, the map
    \begin{align*}
      \phi: ~~ A &\to A/\sim\\
      (i, j)&\mapsto (\psi(i), \psi(j))
  \end{align*}
is a bijection. 
\end{lem}
\begin{proof}
    We first address the well-definition. If $(i, j)$ is an arc of $A$, then $i$ and $j$ lie in distinct components of the missing forest $\missing{A}$. Hence $\psi(i) \neq \psi(j)$ and $(\psi(i), \psi(j))$ is an arc of the quotient arrowflow $A/\sim$. Thus $\phi$ is well-defined.

    Surjectivity follows from the definitions. To show injectivity, suppose that $\phi(a) = \phi(b)$ for two arcs $a = (u_a, v_a)$ and $b = (u_b, u_b)$ of $A$:
     \[
    \begin{tikzpicture}[x = 1em, y = 1em]
        \node (v1) at (0,0) {$u_a$};
        \node (v2) at (0,1) {$u_b$};
        \node (v3) at (3,0) {$v_a$};
        \node (v4) at (3,1) {$v_b$};
        \draw[->] (v1) to["$a$"'] (v3);
        \draw[->] (v2) to["$b$"] (v4);
    \end{tikzpicture}
    \]
    Since $\phi(a) = \phi(b)$, the vertices $u_a$ and $u_b$ are mapped to the same vertex under the quotient, so we deduce $u_a$ and $u_b$ belong to the same connected component of $\missing{A}$.
    In particular, there is a path $P(u_a, u_b) \subseteq \missing{A} \subseteq T$.
    Similarly, there is a path $P(v_b, v_a) \subseteq \missing{A} \subseteq T$. But then, there are two paths
    \[
    P(u_a, u_b) ~ b ~ P(v_b, v_a) 
    \quad\text{and}\quad
     a
    \]
    from $u_a$ to $v_a$ in $T$.
    Since $T$ is a tree, this implies $u_a = u_b$ and $v_a = v_b$. We conclude that $\phi$ is a bijection, as claimed.
\end{proof}

If $A$ is a unital arrowflow, then its missing forest is $S$-rooted. That is, the connected components of $\missing{A}$ are indexed by the vertices of $S$. We therefore take $I = S$. Similarly, if $A$ is a composite arrowflow, then its missing forest $\missing{A}$ is $(S,*)$-rooted; we can take $I = S\sqcup\{*\}$, where $*$ denotes the quotient projection of the floating component of $\missing{A}$. Remark that, in both cases, the restriction of the quotient map $\psi:[n]\to I$ to $S$ is the identity map, $\psi|_S = \text{Id}_S$.

\begin{lem}\label{lem: quotient type}
    Let $A$ be an $m$-arrowflow on $T$. Then its quotient arrowflow $A/\sim$ is an $m$-arrowflow on $T/\sim$. Moreover,
    \begin{enumerate}
        \item $A$ has parallel arrows if and only if $A/\sim$ has parallel arrows,
        \item if $A$ is unital (for $S$) then $A/\sim$ is unital (for $S$), and
        \item if $A$ is composite (for $S$) then $A/\sim$ is composite (for $S$).
    \end{enumerate}
\end{lem}
\begin{proof}
    Let $A$ be an $m$-arrowflow on $T$.
    The vertex sets of both $A/\sim$ and $T/\sim$ coincide. Furthermore, $A/\sim$ has $m$ arcs counted with multiplicity and 
    \[
    \SG(A/\sim) ~=~ \faktorDisplay{\SG(A)}{\sim} ~\subseteq~ \faktor{T}{\sim}\,.
    \]
    Thus, $A/\sim$ is an $m$-arrowflow on $T/\sim$. 

    If $A$ has parallel arrows, so does $A/\sim$. Reciprocally, if $(u,v)$ appears with multiplicity at least two in $A/\sim$, then there are two arcs joining the components $\psi^{-1}(u)$ and $\psi^{-1}(v)$ of $\missing{A}$. As argued in Lemma \ref{lem: quotient arc bijection}, this implies that the arcs are supported on the same edge of $A$, and pointing in the same direction. This shows (1).

    We show (2) and (3) together. If $A$ is unital or composite, the vertex set of  $A/\sim$ is $I = S$ or $I = S \sqcup \{*\}$, respectively. The result follows noting that, in either case, every vertex is isolated in $\missing{A/\sim}$. 
\end{proof}

\begin{ej}
    The arrowflow of Figure \ref{fig: quotient unital c} is unital and it is the quotient arrowflow of the unital arrowflow of Figure \ref{fig: arrowflows c}.
    The arrowflow of Figure \ref{fig: quotient composite c} is composite and it is the quotient arrowflow of the composite arrowflow of Figure \ref{fig: arrowflows d}.
\end{ej}

\begin{figure}[h]
    \centering
    \begin{subfigure}{0.32\textwidth}
        \centering
        \includegraphics[page = 1]{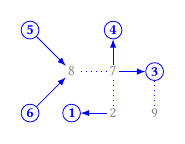}
        \subcaption{An arrowflow $A$.}
        \label{fig: quotient composite a}
    \end{subfigure}
    \begin{subfigure}{0.32\textwidth}
        \centering
        \includegraphics[page = 2]{Figures/quotientComposite.pdf}
        \subcaption{Its missing forest $\missing{A}$.}
        \label{fig: quotient composite b}
    \end{subfigure}
    \begin{subfigure}{0.32\textwidth}
        \centering
        \includegraphics[page = 3]{Figures/quotientComposite.pdf}
        \subcaption{The quotient arrowflow $A/\sim$.}
        \label{fig: quotient composite c}
    \end{subfigure}
    \caption{ }
    \label{fig: quotient composite}
\end{figure}

\begin{lem}
\label{lem:C_S(A) sum equals C_S(A/sim)}
  Let $A$ be either a unital or a composite $m$-arrowflow. There is a sign-preserving bijection between $C_S(A)$ and $C_S(\faktor{A}{\sim})$. In particular, 
  \[
  \sum_{\kappa \in C_S(A)} \sgn(\kappa) = \sum_{\kappa \in C_S(A/\sim)} \sgn(\kappa).
  \]
\end{lem}
\begin{proof}
Let $\phi$ be the bijection of Lemma \ref{lem: quotient arc bijection}.
Given a $S$-catalyst $(\sigma, f)$ of $C_S(A)$, we construct an $S$-catalyst $(\Phi\sigma, \Phi f)$ in the catalyst class $C_S(\faktor{A}{\sim})$ by letting $\Phi\sigma = \sigma$ and $\Phi f = \phi \circ f$. 
The map $(\sigma, f) \mapsto (\Phi \sigma, \Phi f)$ has an inverse $(\tau, g)\mapsto(\Psi\tau, \Psi g)$ and is given by  $\Psi\tau := \tau$ and $\Psi g := \phi^{-1}\circ g$. Thus, it is a bijection. That it is sign-preserving is immediate. 
\end{proof}

\section{Route maps}
\label{sec:route maps}

To provide combinatorial proofs of Propositions \ref{prop:unital} and \ref{prop:composite}, we transform our $S$-catalyst enumeration problem into a path enumeration problem in a network that we call the route map. This latter class of problems are well-studied and is of extreme beauty. The cornerstone of this field lies in the following lemma.

\begin{lem}[LGV lemma, \cite{Lindstrom} \cite{Gessel-Viennot}]
    \label{lem: LGV}
For any  network $\mathcal{G}$,
there exists a sign-reversing involution on the set of all families of network paths of $\mathcal{G}$ such that
\begin{enumerate}
\item its fixed points are the non-intersecting families of network paths, and
\item it stabilizes the multiset of steps of the families of network paths.
\end{enumerate}
\end{lem}

We introduce the necessary concepts to understand this lemma below.
The proof is elegant and well-known, even featured in \textit{Proofs from \textsc{the book}} \cite{theBook}, and so we omit it. We remark it is bijective.

Then, we recall the definition of the route map of an arrowflow introduced in \cite{FPSAC, BEGLR}. The definition is involved and studying its properties represents the bulk of said article. In this section, we introduce the definition and collect its main properties, referring to \cite{BEGLR} for many of the proofs.

\subsection{Generalities about networks}
\begin{de}[Network]
    A \emph{network} is a tuple
    \(
    \big(
    \mathcal{G},
    \Delta,
    \nabla
    \big)
    \)
    consisting of an acyclic digraph $\mathcal{G}$ and two sets of distinguished nodes of $\mathcal{G}$ such that $\#\Delta = \#\nabla$. We order these sets, 
    $\Delta = (i_1, ..., i_n)$ and $\nabla = (j_1, ..., j_n)$. We say each node in $\Delta$ is a \emph{source} and each node in $\nabla$ is a \emph{sink} of the network.
\end{de}

A \emph{network path} in $\mathcal{G}$ is a path from a source $i\in\Delta$ to a sink $j \in \nabla$. A \emph{family of network paths} in $\mathcal{G}$ is a tuple $\Lambda = (\Lambda_i)_{i\in \Delta}$ of network paths, each starting from a source $i\in\Delta$ and ending on a distinct sink. A family of network paths is non-intersecting if paths are pairwise vertex-disjoint. 
The \emph{underlying map of a family $\Lambda$} is the bijection $\sigma : \Delta \to
\nabla$ such that $\Lambda_i$ is a network path from $i$ to $\sigma(i)$ for all $i \in \Delta$. When there is a canonical identification of $\Delta$ and $\nabla$, we refer to this as the underlying permutation, and identify it with an element of $\Sym_\Delta$. 
In this context, we define the sign of $\Lambda$ as the sign of $\sigma\in\Sym_\Delta$.

We let \emph{$\FNP(\mathcal{G})$} denote the set of families of network paths of $\mathcal{G}$.

 We will need a notion of composition of maps $\sigma$ and $\tau$ with different domains and co-domains: let $\sigma : \Delta \to \nabla$ and $\tau : \Delta' \to \nabla'$ be two maps and define $\tau \circ \sigma : \Delta \cup (\Delta'\setminus\nabla) \to (\nabla \setminus \Delta') \cup \nabla'$ to be the map given by
\[
\tau \circ \sigma (x) = \begin{cases}
    \tau(\sigma(x)) & \text{if }x\in \Delta ~\text{and}~ \sigma(x)\in \Delta',\\
    \sigma(x) & \text{if }x\in \Delta ~\text{and}~ \sigma(x)\not\in \Delta',\\
    \tau(x) & \text{if }x\not\in \Delta ~\text{and}~ x\in \Delta'.\\
\end{cases}
\]
This can be understood diagrammatically through the next easy example.
\begin{ej}
    Let $\sigma : \{i_1,i_2\}\to\{j_1,j_2\}$ map $i_1\mapsto j_1$ and $i_2\mapsto j_2$. Let $\tau : \{j_2,i_2'\}\to\{j_1',j_2'\}$ map $j_2\mapsto j_1'$ and $i_2\mapsto j_2'$. Then, the composition $\tau\circ\sigma$ is the following:
\[
    \scriptsize
\begin{tikzpicture}[baseline = -6em, x = 1.5em, y = 2.5em, yscale = -1]
    \node (i1) at (1,3) {};
    \node (i2) at (2,3) {};
    \node (j1) at (1,2) {};
    \node (j2) at (2,2) {};
    \node (i2') at (3,2) {};
    \node (j1') at (2,1) {};
    \node (j2') at (3,1) {};
    
    \filldraw (i1) circle (1.5pt) node[below] {$i_1$};
    \filldraw (i2) circle (1.5pt) node[below] {$i_2$};
    \filldraw (j1) circle (1.5pt) node[left] {$j_1$};
    \filldraw (j2) node {$j_2$};
    \filldraw (i2') circle (1.5pt) node[right] {$i_2'$};
    \filldraw (j1') circle (1.5pt) node[above] {$j_1'$};
    \filldraw (j2') circle (1.5pt) node[above] {$j_2'$};

    \begin{scope}[semithick, ->]
        \draw (i1) -- (j1);
        \draw[shorten >= 0.4pt] (i2) -- (j2);
        \draw[shorten <= 0.4pt] (j2) -- (j1');
        \draw (i2') -- (j2');
    \end{scope}
    \end{tikzpicture}
    \qquad = \qquad
\begin{tikzpicture}[baseline = -6em, x = 1.5em, y = 2.5em, yscale = -1]
    \node (i1) at (1,3) {};
    \node (i2) at (2,3) {};
    \node (j1) at (1,1) {};
    \node (i2') at (3,3) {};
    \node (j1') at (2,1) {};
    \node (j2') at (3,1) {};
    
    \filldraw (i1) circle (1.5pt) node[below] {$i_1$};
    \filldraw (i2) circle (1.5pt) node[below] {$i_2$};
    \filldraw (i2') circle (1.5pt) node[below] {$i_2'$};
    \filldraw (j1) circle (1.5pt) node[above] {$j_1$};
    \filldraw (j1') circle (1.5pt) node[above] {$j_1'$};
    \filldraw (j2') circle (1.5pt) node[above] {$j_2'$};

    \begin{scope}[semithick, ->]
        \draw (i1) -- (j1);
        \draw (i2) -- (j1');
        \draw (i2') -- (j2');
    \end{scope}
    \end{tikzpicture}
\]
\end{ej}

So far, we have been using the words sinks and sources in the context of networks, to refer to distinguished points of a graph. In some applications, however, it is useful to require these to be graph-theoretic sinks and sources: a node of a digraph is a \emph{graph-theoretic sink} if it only has incoming arcs, and a \emph{graph-theoretic source} if it only has outgoing arcs.
\begin{lem}\label{lem: union of networks}
     Let $(\mathcal{G}, \Delta, \nabla)$ and $(\mathcal{G}', \Delta', \nabla')$ be two networks. Suppose that $\mathcal{G}\cap\mathcal{G}' = \nabla \cap \Delta'$ and that each node in $\nabla\cap\Delta'$ is a graph-theoretic sink of $\mathcal{G}$ and a graph-theoretic source of $\mathcal{G}'$. Then, 
     \[
     \mathcal{G}\doubleplus\mathcal{G}' :=
     \Big(\mathcal{G}\cup\mathcal{G}', \quad
     \Delta \cup (\Delta' \setminus \nabla),\quad
     (\nabla \setminus \Delta') \cup \nabla'\Big)
     \]
     is a network and
     \(
     \FNP( \mathcal{G}\doubleplus\mathcal{G}' ) = \FNP( \mathcal{G})\times \FNP( \mathcal{G}' ).
     \)
     Moreover, the underlying map of $(\Lambda,\Lambda')$ is $\sigma_{(\Lambda,\Lambda')} = \sigma_{\Lambda'}\circ \sigma_{\Lambda}$.
\end{lem}
We call this the \emph{concatenation} of networks. (Beware: the operation is not symmetric on $\mathcal{G}$ and $\mathcal{G}'$.)
\begin{proof}
    The graph $\mathcal{G}\cup\mathcal{G}'$ is acyclic, since it is the union of acyclic graphs by graph-theoretic sinks of one and graph-theoretic sources of another. That the new sets of sources and sinks are equinumerous is immediate.

    A network path of the concatenation starts either in $\Delta$ or in $\Delta' \setminus \nabla$. If it starts in $\Delta' \setminus \nabla$, then it must end in $\mathcal{G}'$, since it cannot cross back to $\mathcal{G}$. If it starts in $\Delta$, then it must visit a node of $\nabla$, since these separate $\Delta$ from $\nabla'$. Hence every network path splits uniquely as a pair of network paths of $\mathcal{G}$ and $\mathcal{G}'$.
    
    Finally, the underlying map is the composition of the factors, by definition.
\end{proof}

\subsection{Plane rooted directed trees}

Throughout this section, let $T_* = (V\cup\{*\},E_*)$ be a rooted (simple) tree with root $*$ and let $(V\cup\{*\}, A_*)$ be a rooted \emph{directed} tree supported on $T_*$. 
We refer to this latter by $A_*$ by abuse of notation.
We assume $\SG(A_*) = T_*$.

\begin{de}
    Let $x\in V$ be a node with parent $X\in V\cup\{*\}$ in $T_*$. We say that $x$ is an \emph{ascending child} of $X$ if $(x,X) \in A_*$.  
    Otherwise, we have $(X,x)\in A_*$ and we say $x$ is a \emph{descending child}.
\end{de}
For $X\in V\cup\{*\}$ let \emph{$\Asc(X)$} be the set of ascending children of $X$ and let \emph{$\Des(X)$} be the set descending children.
If $x$ is a node with parent $X$, then the set of neighbors of $x$ is $\partial x = \Asc(x) \cup \Des(x) \cup \{X\}$.
We denote by $T_X$ the subtree of $T_*$ formed by $X$ and all of its descendants, and define $A_X$ similarly. 
\begin{de}
    The \emph{plane rooted directed tree structure on $T_*$ induced by $A_*$} is the tuple
    \[
    \Big(A_*, (<_x)_{x\in V\cup\{*\}}\Big)
    \]
    where for each $x\in V\cup\{*\}$ we have a total order $<_x$ on the set of neighbors $\partial x$ such that any ascending child precedes any descending child. That is, such that
    \begin{itemize}
        \item $y <_x z$ for each $y\in\Asc(x)$ and $z\in\Des(x)$, and
        \item $y <_x X$ if $X$ is the parent of $x$, for all $y \in \Asc(x)\cup\Des(x)$.
    \end{itemize}
    (Note that this second condition is void for the root $x = *$.)
    By abuse of notation, we denote this structure by $T_*$. See Figure \ref{fig: plane rooted tree} for an example.
\end{de}

Hereafter, we write $(\partial x, <_x)$ for the totally ordered poset defined by the plane rooted directed tree structure on the set of neighbors $\partial x$, for $x\in V\cup\{*\}$.

\subsection{Hemisphere of a plane rooted directed tree} 
A \emph{hemisphere $\mathcal{H}(T_*)$} is a digraph constructed from a plane rooted tree $T_*$. There are three types of nodes in $\mathcal{H}(T_*)$:  \emph{$v$-nodes}, coming from the vertices of $T_*$;   \emph{$e$-nodes}, coming from oriented edges supported on $T_*$, and   \emph{$s$-nodes}, that represent the oriented sectors of $T_*$.
See Figure \ref{fig: route map}.
 Explicitly, the set of nodes of $\mathcal{H}(T_*)$ consists of:
\begin{enumerate}
    \item one $v$-node $\Delta_i$ for each vertex $i\in V$.
\item two $e$-nodes $e(i,j)$ and $e(j,i)$ for every edge $\{i, j\} \in E_*$.
    \item two sequences of $s$-nodes 
    \begin{align*}s_i(j_1, j_2), s_i(j_2, j_3), \ldots, s_i(j_{m-1}, j_m) \qquad \text{and} \\
    s_i(j_m, j_{m-1}), s_i(j_{m-1}, j_{m-2}), \ldots, s_i(j_2, j_1)
    \end{align*}for each vertex $i \in V\cup\{*\}$, where $(\partial i, <_i) = (j_1 <_i j_2 <_i \cdots <_i j_m)$ are the neighbors of $i$.
\end{enumerate} 
\begin{note}
    The most attentive reader might have spotted that the definitions here presented and those of \cite{BEGLR, FPSAC} vary slightly. In both cases, ``vestigial'' nodes are included in the definition, which have been removed here. For instance, the construction of \cite{FPSAC} has a node $\Delta_*$ that is \emph{not} a source of the network. These play no role in the constructions, and thus the networks are interchangeable.
\end{note}

   \begin{figure}[h] \tiny
\begin{subfigure}[t]{0.24\textwidth}
\includegraphics[page = 1, width=.9\textwidth]{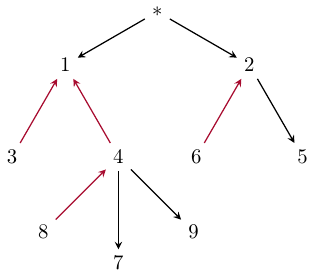}
\caption{A plane rooted tree $T_*$}
\label{fig: plane rooted tree}
\end{subfigure}
\begin{subfigure}[t]{0.24\textwidth}
\includegraphics[page = 2, width=.9\textwidth]{Figures/hemisphere.pdf}
\caption{$v$-nodes of $\mathcal{H}(T_*)$}
\label{fig: route map A}
\end{subfigure}
\hfill
\begin{subfigure}[t]{0.24\textwidth}
\includegraphics[page = 3, width=.9\textwidth]{Figures/hemisphere.pdf}
\caption{$e$-nodes of $\mathcal{H}(T_*)$}
\label{fig: route map B}
\end{subfigure}
\hfill
\begin{subfigure}[t]{0.24\textwidth}
\includegraphics[page = 4, width=.9\textwidth]{Figures/hemisphere.pdf}
\caption{$s$-nodes of $\mathcal{H}(T_*)$}
\label{fig: route map C}
\end{subfigure}
\caption{Nodes of $\mathcal{H}(T_*)$.}
    \label{fig: route map}
\end{figure}

The arcs of $\mathcal{H}(T_*)$ are better described in terms of local digraphs $\Gamma(i)$, one for each vertex $i\in V\cup\{*\}$.  If the neighbors of $i$ are  $(\partial i, <_i) = (j_1 <_i j_2 <_i \cdots <_i j_m)$, then refer to Figure \ref{Gamma i} for an illustration of $\Gamma(i)$. In order to obtain $\Gamma(*)$, let $i = *$ in Figure \ref{Gamma i} and remove $\Delta_*$.

\begin{figure}
\[\scriptsize\begin{tikzcd}[cramped, column sep = tiny] 
  {\Delta_i}
  &  &  {s_i(j_1,j_2)}
  &  &  {s_i(j_2,j_3)}  
  &  &  {s_i(j_3, j_4)}
  &  &  {s_i(j_4, j_5)}
  &  \\
  & e(j_1, i) 
  &  & e(j_2, i) 
  &  & e(j_3, i) 
  &  & e(j_4, i) 
  &  & e(j_5, i) 
  \\
  & \textcolor{gray}{
  e(i, j_1) 
  }
  &  &  \textcolor{gray}{
 e(i, j_2) 
  }
  &  &  \textcolor{gray}{
  e(i, j_3) 
  }
  &  &  \textcolor{gray}{
  e(i, j_4) 
  }
  &  &  \textcolor{gray}{
  e(i, j_5) 
  }
  \\
  &  &  {s_i(j_2,j_1)}
  &  &  {s_i(j_3,j_2)}
  &  &  {s_i(j_4, j_3)}
  &  &  {s_i(j_5, j_4)}
  &
  \arrow[from=1-1, to=3-2, bend right, gray]
  \arrow[from=1-1, to=1-3]
  \arrow[from=1-3, to=1-5]
  \arrow[from=1-5, to=1-7]
  \arrow[from=1-7, to=1-9]
  %
  \arrow[from=2-2, to=1-3]
  \arrow[from=2-4, to=1-5]
  \arrow[from=2-6, to=1-7]
  \arrow[from=2-8, to=1-9]
  %
  \arrow[to=3-4, from=1-3, bend right, gray]
  \arrow[to=3-6, from=1-5, bend right, gray]
  \arrow[to=3-8, from=1-7, bend right, gray]
   \arrow[to=3-10, from=1-9, bend right, gray]
    %
  \arrow[from=2-4, to=4-3, crossing over, bend right]
  \arrow[from=2-6, to=4-5, crossing over, bend right]
  \arrow[from=2-8, to=4-7, crossing over, bend right]
  \arrow[from=2-10, to=4-9, crossing over, bend right]
  %
  \arrow[to=3-2, from=4-3, gray]
  \arrow[to=3-4, from=4-5, gray]
  \arrow[to=3-6, from=4-7, gray]
  \arrow[to=3-8, from=4-9, gray]
   %
  \arrow[to=4-3, from=4-5]
  \arrow[to=4-5, from=4-7]
  \arrow[to=4-7, from=4-9]
\end{tikzcd}
\]

\caption{Digraph $\Gamma(i)$ for a vertex  $i\in V$ and with neighbors $(\partial i, <_i) = (j_1 <_i j_2 <_i j_3 <_i j_4 <_i j_5)$.}
\label{Gamma i}
\end{figure}
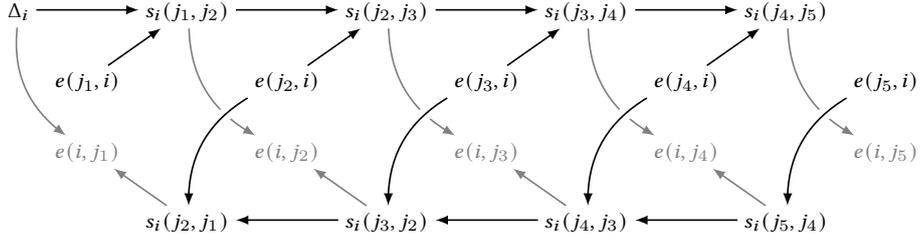

A key property of the local graphs  is that the local digraphs $\Gamma(i)$ and $\Gamma(j)$  intersect if and only if $i$ and $j$ are neighbors in $T_*$. In this situation, the  intersection uniquely consists of the pair of nodes ${e(i,j), e(j, i)}$.

\begin{p}
    The hemisphere is a network
    \(
    \big(
    \mathcal{H}(T_*), (\Delta_i)_{i\in V}, (e(a))_{a\in A_*}
    \big).
    \)
\end{p}
\begin{proof}
    Note that $\#V = \# A_*$.
    That $\mathcal{H}(T_*)$ is acyclic is \cite[Prop.~6.6]{BEGLR}.
\end{proof}

\subsection{The route map}
The \emph{route map $\mathcal{R}$} of a plane rooted directed tree $A_*$ is created from two hemispheres. The \emph{Southern hemisphere} is \emph{$\mathcal{S}$} $:= \mathcal{H}(T_*)$. 

Let $A_*' = \{(j,i) \ : \ (i,j)\in A_*\}$  be the mirror of $A_*$ and define the mirror local order $<'_x$ by letting $y <'_x z$ if and only if $z <_x y$, for each $x\in V\cup\{*\}$.
Let $T'_* = \big(A_*', (<'_x)_{x\in V\cup\{*\}}\big)$ be the plane rooted directed tree structure induced by $A_*'$, and let $\mathcal{H}(T'_*)$ be its hemisphere. The \emph{Northern hemisphere $\mathcal{N}$} is created from $\mathcal{H}(T'_*)$ by reversing each arc in $\mathcal{H}(T'_*)$ and relabeling each node $x$ by $\Psi(x)$:
    \begin{align*}\Psi(\Delta_i) = \nabla_i,\
    \Psi(e(i, j))=e'(j,i),\
    \Psi(s_i(u, v))= s_i'(v, u).
    \end{align*}
We define a network $\mathcal{B}$ whose vertex set is $\{e(a)~:~a\in A_*\}\cup\{e'(a)~:~a\in A_*\}$, and with $n$ arcs $(e(a), e'(a))$, one for each $a$ in $A_*$. We let $e(a)$ be a source and $e'(a)$ be a sink for each $a\in A_*$.
The route map $\mathcal{R}$ is defined as the concatenation
\[
\mathcal{S} \doubleplus \mathcal{B} \doubleplus \mathcal{N}.
\]
The arcs of $\mathcal{B}$ are referred to as the \emph{bridges between hemispheres} of $\mathcal{R}$. Note that there is a unique family of network paths in $\mathcal{B}$, and that its underlying map is $\iota: e(a)\mapsto e'(a)$ for all $a\in A_*$.

\begin{cor}
    The route map inherits a network structure
    \(
    \big(
    \mathcal{R}, (\Delta_i)_{i\in V}, (\nabla_i)_{i\in V}
    \big).
    \)
\end{cor}

A family of network paths $\Lambda \in \FNP({\mathcal{R}})$ is called \emph{full} if it visits every bridge between hemispheres. Note that if $\Lambda$ is non-intersecting, it is in particular full.

\begin{p}[\cite{BEGLR}]
\label{prop:lift paths to routemap}
    Let $T_* = (V\cup\{*\},E_*)$ be a rooted tree with a plane directed rooted tree structure induced by a directed tree $A_*$.
    Let $\mathcal{R}$ be the route map of $A_*$.
    \begin{enumerate}
        \item \label{propItem:paths in routemap}
        Let $i, j \in V$, let $a\in P(i,j)\subseteq A_*$.
        \begin{enumerate}
            \item There is a unique path from $\Delta_i$ to $e(a)$ in $\mathcal{S}$.
            \item There is a unique path from $e'(a)$ to $\nabla_j$ in $\mathcal{N}$.
            \item There is a unique path from $\Delta_i$ to $\nabla_j$ in $\mathcal{R}$.
        \end{enumerate}
        Furthermore, every network path of $\mathcal{R}$ arises this way.
        \item \label{propItem:bijection catalysts routemap}
        The correspondence of (1) induces a bijection between the set $C(A_*)$ of $V$-catalysts inducing $A_*$ and the set of full families of network paths in $\mathcal{R}$.
    \end{enumerate}
\end{p}
\begin{proof}
    We point at the precise results in \cite{BEGLR} to guide the more interested reader. The first part follows from Lemma 6.5 as deduced in Proposition 6.6 and Corollary 6.7. The second part is Lemma 6.10.
\end{proof}

\subsection{Local network structure}
The network structure of $\mathcal{R}$ can also be constructed locally, by endowing each local graph $\Gamma(i)$, $i\in V\cup\{*\}$, with a network structure.
\begin{de}
    Let $x \in V$, let $\partial x = \Asc(x)\cup \Des(x) \cup \{X\}$.
    Let $\Gamma_{\mathcal{S}}(x)$ be the local graph of $x$ in $\mathcal{S}$. Endow it with network structure by setting sources
    \[
    \{\Delta_x\} \cup \{e(y, x) \ :\  y\text{ child of }x\}
    \]
    and sinks  
    \[
    \{e(x, X)\} \cup \{e(a)\ : \ a\in A_*\text{ adjacent to }x,\, a\ne(X,x)\}.
    \]
    Let $\Gamma_{\mathcal{N}}(x)$ be the local graph of $x$ in $\mathcal{N}$. Endow it with a network structure with sources
    \[
    \{e'(X,x)\} \cup \{e'(a)\ : \ a\in A_*\text{ adjacent to }x\}.
    \]
    and sinks  
    \[
    \{\nabla_x\} \cup \{e'(x, y) \ :\  y\text{ child of }x\}
    \]
    We call these the \emph{Southern (resp. Northern) local network of $x$}.
\end{de}

It is apparent from Figure \ref{Gamma i} that these are acyclic graphs. The sets of sources and sinks of the Southern and Northern local network have $\delta(x)$ many nodes each.

The local graph around the root is slightly different.
    Let $\Gamma_{\mathcal{S}}(*)$ be the local graph of $*$ in $\mathcal{S}$. We endow it with a network structure with
    \begin{enumerate}
        \item a source $e(y,*)$ for each $y\in\partial*$, and
        \item a sink $e(a)$ for each $a \in A_*$  adjacent to $*$.
    \end{enumerate}
    Let $\Gamma_{\mathcal{N}}(*)$ be the local graph of $*$ in $\mathcal{N}$ and endow it with a network structure with
    \begin{enumerate}
        \item a sink $e'(*,y)$ for each $y\in\partial*$, and
        \item a source $e'(a)$ for each $a \in A_*$ adjacent to $*$.
    \end{enumerate}
    We call these the \emph{Southern (resp. Northern) local network of the root $*$}. 

    Let $\mathcal{B}(x)$ be the network with one source $e(a)$ and one sink $e'(a)$ for each arc $a$ adjacent to $x$ (except for $a = (x,X)$), no other nodes, and arcs $\big(e(a), e'(a)\big)$.
\begin{de}
    The \emph{local network of $x$} is \emph{$\Gamma_{\mathcal{R}}(x)$}$:=\Gamma_{\mathcal{S}}(x)\doubleplus\mathcal{B}(x)\doubleplus\Gamma_{\mathcal{N}}(x)$.
\end{de}
    See Figure \ref{fig: non-intersecting local network x}. We spell out the network structure of the local network. Let $x \in V$ be a vertex.
    Then $\Gamma_{\mathcal{R}}(x)$ $x$ has sources
    \(
    \{\Delta_x\}\cup\{e(y,x)\ :\ y\text{ child of }x\}\cup\{e'(X,x)\}
    \)
    and sinks
    \(
    \{\nabla_x\}\cup\{e'(x,y)\ :\ y\text{ child of }x\}\cup\{e(x,X)\}.
    \)
    
    Let $x = *$ be the root. Then $\Gamma_{\mathcal{R}}(*)$ has
    \begin{enumerate}
        \item a source $e(y,*)$ for each child $y$ of $*$, and
        \item a sink $e'(*,y)$ for each child $y$ of $*$.
    \end{enumerate}

We now study non-intersecting families of network paths in $\mathcal{R}$. Recall that $\mathcal{R} = \mathcal{S} \doubleplus \mathcal{B} \doubleplus \mathcal{N}$. We will shortly refine this concatenation, to allow us to compute the underlying map of any family of network paths.

\begin{lem}
\label{lem: underlying permutation south}
Let $x \in V$ be a vertex with $a = \#\Asc(x)$ ascending children and $d = \#\Des(x)$ descending children. Let 
\[
(\partial x, <_x) = (y_1 <_x \cdots <_x y_{a} <_x z_1 <_x \cdots <_x z_{d} <_x X).
\]
There is a unique non-intersecting family of network paths in $\Gamma_{\mathcal{S}}(x)$ and its underlying map is the following:
\[
\sigma_x :=
    \scriptsize
    \begin{tikzpicture}[baseline = -1.5em, x = 3.5em, y = 3em, yscale = -1]
    \foreach \i in {2, ..., 9}{
    \node (\i) at (\i,0) {};
    }
    \foreach \i in {2,3,4,6,7,8}{
    \filldraw (\i) circle (1.5pt);
    }
    \filldraw (2) node[above] {$e(x,y_1)$} ++(0,1) node[below] {$e(x,y_1)$} circle (1.5pt);
    \filldraw (3) node[above] {$\cdots$} ++(0,1) node[below] {$\cdots$} circle (1.5pt);
    \filldraw (4) node[above] {$e(x,y_a)$} ++(0,1) node[below] {$e(x,y_a)$} circle (1.5pt);
    \filldraw (5)++(0,1) circle (1.5pt) node[below] {$\Delta_x$} circle (1.5pt);
    \filldraw (6) node[above] {$e(x,z_1)$} ++(0,1) node[below] {$e(z_1,x)$} circle (1.5pt);
    \filldraw (7) node[above] {$\cdots$} ++(0,1) node[below] {$\cdots$} circle (1.5pt);
    \filldraw (8) node[above] {$e(x,z_d)$} ++(0,1) node[below] {$e(z_d,x)$} circle (1.5pt);
    \filldraw (9) circle (1.5pt) node[above] {\,\qquad$e(x,X)$};

    \begin{scope}[semithick, ->]
        \draw (2)++(0,1) -- (2);
        \draw (3)++(0,1) -- (3);
        \draw (4)++(0,1) -- (4);
        \draw (5)++(0,1) -- (6);
        \draw (6)++(0,1) -- (7);
        \draw (7)++(0,1) -- (8);
        \draw (8)++(0,1) -- (9);
    \end{scope}
    \end{tikzpicture}
\]
\end{lem}
\begin{proof}
    This is \cite[Lemma 7.5]{BEGLR} (uniqueness and map) and \cite[Proposition 8.9]{BEGLR} (existence). See the southern part of Figure \ref{fig: non-intersecting local network x} for an example.
\end{proof}
    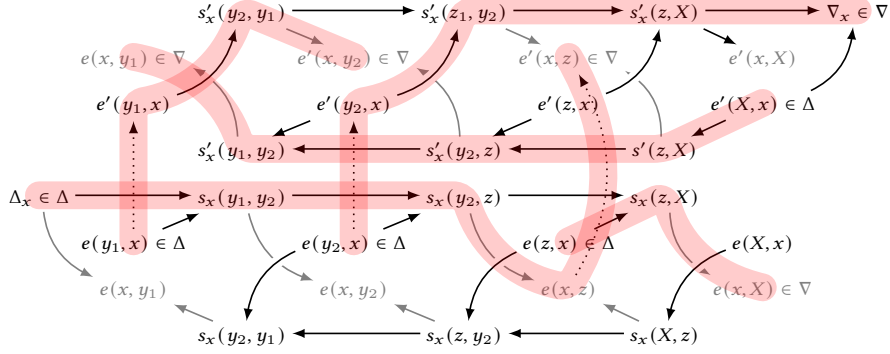
\begin{figure}
\[\tiny
\begin{tikzcd}[cramped, column sep = 0pt, row sep = small, every matrix/.append style={name=Coord}, remember picture]
	&& {s'_x(y_2,y_1)} && {s'_x(z_1,y_2)} && {s'_x(z,X)} && {\nabla_x\in\nabla} \\
	& \textcolor{gray}{e(x,y_1)\in\nabla} && \textcolor{gray}{e'(x,y_2)\in\nabla} && \textcolor{gray}{e'(x,z)\in\nabla} && \textcolor{gray}{e'(x,X)} \\
	& {e'(y_1,x)} && {e'(y_2,x)} && {e'(z,x)} && {e'(X,x)\in\Delta} \\
	  && {s'_x(y_1,y_2)} && {s'_x(y_2,z)} && {s'(z,X)} \\
	{\Delta_x\in\Delta} && {s_x(y_1,y_2)} && {s_x(y_2,z)} && {s_x(z,X)} \\
	& {e(y_1,x)\in\Delta} && {e(y_2,x)\in\Delta} && {e(z,x)\in\Delta} && {e(X,x)} \\
	& \textcolor{gray}{e(x,y_1)} && \textcolor{gray}{e(x,y_2)} && \textcolor{gray}{e(x,z)} && \textcolor{gray}{e(x,X)\in\nabla} \\
	&& {s_x(y_2,y_1)} && {s_x(z,y_2)} && {s_x(X,z)}
	\arrow[from=1-3, to=1-5]
	\arrow[gray, from=1-3, to=2-4]
	\arrow[from=1-5, to=1-7]
	\arrow[from=1-7, to=1-9]
	\arrow[gray, from=1-5, to=2-6]
	\arrow[from=1-7, to=2-8]
	\arrow[from=3-4, to=4-3]
	\arrow[from=3-6, to=4-5]
	\arrow[from=3-8, to=4-7]
	\arrow[gray, bend right, from=4-3, to=2-2]
	\arrow[gray, bend right, from=4-5, to=2-4]
	\arrow[from=4-5, to=4-3]
	\arrow[gray, bend right, from=4-7, to=2-6]
	\arrow[from=4-7, to=4-5]
	\arrow[from=5-1, to=5-3]
	\arrow[gray, bend right, from=5-1, to=7-2]
	\arrow[from=5-3, to=5-5]
	\arrow[gray, bend right, from=5-3, to=7-4]
	\arrow[from=5-5, to=5-7]
	\arrow[gray, bend right, from=5-5, to=7-6]
	\arrow[gray, bend right, from=5-7, to=7-8]
	\arrow[dotted, from=6-2, to=3-2]
	\arrow[from=6-2, to=5-3]
	\arrow[dotted, from=6-4, to=3-4]
	\arrow[from=6-4, to=5-5]
	\arrow[bend right, from=6-4, to=8-3, crossing over]
	\arrow[from=6-6, to=5-7]
	\arrow[bend right, from=6-6, to=8-5, crossing over]
	\arrow[bend right, from=6-8, to=8-7, crossing over]
	\arrow[bend right, dotted, from=7-6, to=2-6]
	\arrow[gray, from=8-3, to=7-2]
	\arrow[gray, from=8-5, to=7-4]
	\arrow[from=8-5, to=8-3]
	\arrow[gray, from=8-7, to=7-6]
	\arrow[from=8-7, to=8-5]
	\arrow[bend right, from=3-2, to=1-3, crossing over]
	\arrow[bend right, from=3-4, to=1-5, crossing over]
	\arrow[bend right, from=3-6, to=1-7, crossing over]
	\arrow[bend right, from=3-8, to=1-9, crossing over]
\end{tikzcd}
\begin{tikzpicture}[overlay, remember picture]
    \begin{scope}[red, line width = 1.5em, opacity = 0.2, rounded corners, line cap=round]
    \draw (Coord-5-1.center) -- (Coord-5-3.center) -- (Coord-5-5.center) to[bend right] (Coord-7-6.center) to[bend right] (Coord-2-6.center);
    \draw (Coord-6-2.center) -- (Coord-3-2.center) to[bend right] (Coord-1-3.center) -- (Coord-2-4.center);
    \draw (Coord-6-4.center) -- (Coord-3-4.center) to[bend right] (Coord-1-5.center) -- (Coord-1-9.center);
    \draw (Coord-6-6.center) -- (Coord-5-7.center) to[bend right]  (Coord-7-8.center);
    \draw (Coord-3-8.center) -- (Coord-4-7.center) -- (Coord-4-3.center) to[bend right] (Coord-2-2.center);
    \end{scope}
\end{tikzpicture}
\]
\caption{Non-intersecting family of network paths in $\Gamma_{\mathcal{R}}(x)$ for a vertex $x\in V$ with two ascending and one descending neighbors.}
\label{fig: non-intersecting local network x}
\end{figure}
\begin{lem}
\label{lem: underlying permutation north}
Let $x \in V$ be a vertex with $a = \#\Asc(x)$ ascending children and $d = \#\Des(x)$ descending children. Let 
\[
(\partial x, <_x) = (y_1 <_x \cdots <_x y_{a} <_x z_1 <_x \cdots <_x z_{d} <_x X).
\]
There is a unique non-intersecting family of network paths in $\Gamma_{\mathcal{N}}(x)$ and its underlying map is the following:
\[
\nu_x :=
    \scriptsize
    \begin{tikzpicture}[baseline = -1.5em, x = 3.5em, y = 3em, yscale = -1]
    \foreach \i in {1, ..., 8}{
    \node (\i) at (\i,0) {};
    }
    \foreach \i in {2,3,4,6,7,8}{
    \filldraw (\i) circle (1.5pt);
    }
    \filldraw (1)++(0,1) circle (1.5pt) node[below] {$e'(X,x)$\qquad\,} circle (1.5pt);
    \filldraw (2) node[above] {$e'(x,y_1)$} ++(0,1) node[below] {$e'(y_1,x)$} circle (1.5pt);
    \filldraw (3) node[above] {$\cdots$} ++(0,1) node[below] {$\cdots$} circle (1.5pt);
    \filldraw (4) node[above] {$e'(x,y_a)$} ++(0,1) node[below] {$e'(y_a,x)$} circle (1.5pt);
    \filldraw (5) circle (1.5pt) node[above] {$\nabla_x$};
    \filldraw (6) node[above] {$e'(x,z_1)$} ++(0,1) node[below] {$e'(x,z_1)$} circle (1.5pt);
    \filldraw (7) node[above] {$\cdots$} ++(0,1) node[below] {$\cdots$} circle (1.5pt);
    \filldraw (8) node[above] {$e'(x,z_d)$} ++(0,1) node[below] {$e'(x,z_d)$} circle (1.5pt);

    \begin{scope}[semithick, ->]
        \draw (1)++(0,1) -- (2);
        \draw (2)++(0,1) -- (3);
        \draw (3)++(0,1) -- (4);
        \draw (4)++(0,1) -- (5);
        \draw (6)++(0,1) -- (6);
        \draw (7)++(0,1) -- (7);
        \draw (8)++(0,1) -- (8);
    \end{scope}
    \end{tikzpicture}
\]
\end{lem}

Let $\iota_x$ be the underlying permutation of the unique family of network paths of $\mathcal{B}(x)$, which sends $e(a)\mapsto e'(a)$ for each $a$ adjacent to $x$.
By Lemma \ref{lem: union of networks}, we conclude that there is a unique non-intersecting family of network paths in \(\Gamma_{\mathcal{R}}(x)\)
and whose underlying permutation is $\nu_x\circ\iota_x\circ\sigma_x$. See Figure \ref{fig: non-intersecting local network x}.
Lemma \ref{lem:underlying map} explains how to pass from the underlying permutation of families of network paths in the local networks to the underlying permutation of families of network paths in the route map. 
    Let $x\in V$ be a vertex with children
    \[
    y_1 <_x \cdots <_x y_{k}.
    \]
    If $x$ is a leaf, let $\overline{\sigma}_x = \sigma_x$ and $\overline{\nu}_x = \nu_x$; otherwise define recursively
    \[
    \overline{\sigma}_x := \sigma_x \circ \overline{\sigma}_{y_1} \circ\cdots\circ \overline{\sigma}_{y_{k}}\,,
    \quad\text{and}\quad
    \overline{\nu}_x := \overline{\nu}_{y_{k}} \circ\cdots\circ \overline{\nu}_{y_1} \circ \nu_x\,.
    \]
\begin{lem}\label{lem:underlying map}
    Let $\Lambda \in \FNP(\mathcal{R})$ be non-intersecting, and let $\rho_*$ be the underlying map of the induced family of network paths in $\Gamma_\mathcal{R}(*)$. Then, the underlying map of $\Lambda$ is
    \[
    \overline{\nu}_{j_{k}} \circ\cdots\circ \overline{\nu}_{j_1} \circ \rho_* \circ \overline{\sigma}_{j_1} \circ\cdots\circ \overline{\sigma}_{j_{k}}\,,
    \]
    where $(\partial*, <_*) = j_1 <_* \cdots <_* j_k$.
\end{lem}
\begin{proof}
    Let $x\in V$ be a vertex with parent $X$. Then $\mathcal{G} = \Gamma_\mathcal{S}(x)$ and $\mathcal{G}' = \Gamma_\mathcal{S}(X)$ are in the hypotheses of Lemma \ref{lem: union of networks}, and therefore we can concatenate them. Similarly, we can concatenate $\Gamma_\mathcal{N}(X)$ and $\Gamma_\mathcal{N}(x)$. 
    Hence, we can construct $\mathcal{R}$ recursively as follows:
    \begin{enumerate}
        \item If $x$ is a leaf, let $\Sigma_\mathcal{S}(x) = \Gamma_\mathcal{S}(x)$ and $\Sigma_\mathcal{N}(x) = \Gamma_\mathcal{N}(x)$.
        \item Otherwise, let 
        \begin{align*}
            \Sigma_\mathcal{S}(x) &= \Sigma_\mathcal{S}(y_1) \doubleplus \cdots \doubleplus \Sigma_\mathcal{S}(y_k) \doubleplus\Gamma_\mathcal{S}(x) \quad\text{and}\\
            \Sigma_\mathcal{N}(x) &= \Gamma_\mathcal{N}(x)\doubleplus \Sigma_\mathcal{N}(y_k) \doubleplus \cdots \doubleplus \Sigma_\mathcal{N}(y_1),
        \end{align*}
        where $y_1, ..., y_k$ are the children of $x$.
        \item Let $\mathcal{R} = \Sigma_\mathcal{S}(*) \doubleplus \mathcal{B} \doubleplus \Gamma_\mathcal{N}(*)$.
    \end{enumerate}
    (Note that $\Sigma_\mathcal{S}(*) = \mathcal{S}$ and $\Sigma_\mathcal{N}(*) = \mathcal{N}$.)
    By applying Lemmas \ref{lem: union of networks}, \ref{lem: underlying permutation south}, and \ref{lem: underlying permutation north}, we conclude.
\end{proof}

\begin{ej}
    Conceptually, Lemma \ref{lem: union of networks} allows us to take the concatenation of networks but only if ``one is pointing into the other one''. Let $T$ be a rooted tree with four nodes: the root $*$, a child $x$, and two grandchildren $y, z$. Choose an arrowflow and construct $T_*$ and $\mathcal{R}$ accordingly. The local networks interact as follows:
    \begin{center}
        \includegraphics[scale = .7]{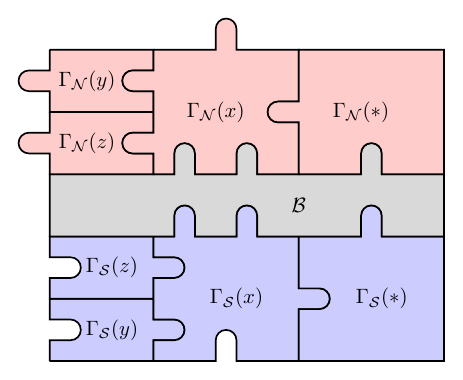} \hspace{2cm}
        \begin{tikzpicture}[x = 3em, y = 2em, baseline = -4em]
            \node (*) at (3,1) {$*$};
            \node (x) at (2,1) {$x$};
            \node (y) at (1,0) {$y$};
            \node (z) at (1,2) {$z$};
            \draw (*) -- (x);
            \draw (x) -- (y);
            \draw (x) -- (z);
        \end{tikzpicture}
    \end{center}
    In the picture, a source is indicated with an indent and a sink with a tab. For instance, $\Gamma_{\mathcal{S}}(x)$ has three sources $e(y,x)$, $e(z,x)$, and $\Delta_x$.
    Concatenating in the following order
    \[
    \Gamma_\mathcal{S}(z) \doubleplus 
    \Gamma_\mathcal{S}(y) \doubleplus 
    \Gamma_\mathcal{S}(x) \doubleplus 
    \Gamma_\mathcal{S}(*) \doubleplus 
    \mathcal{B} \doubleplus 
    \Gamma_\mathcal{N}(*) \doubleplus 
    \Gamma_\mathcal{N}(x) \doubleplus 
    \Gamma_\mathcal{N}(y) \doubleplus 
    \Gamma_\mathcal{N}(z)
    \]
    ensures that at every step one is joining two networks with one ``pointing into'' the next one.
\end{ej}

\subsection{Depth-First-Search walk}
Define the \emph{marked Depth-First-Search walk on $T_*$} recursively by letting $\DFS(x)$ be the trivial walk $\circled{x}$ if $x$ is a leaf, and for each $x\in V$, letting $\DFS(x)$ be given by
\[
x ~ \DFS(y_1) ~ x ~ \DFS(y_2) ~ x~ \cdots~ x \DFS(y_a) ~ \circled{x} ~ \DFS(z_1) ~ x \cdots ~ x ~ \DFS(z_d) ~ x,
\]
where $y_1, ..., y_a$ are the ascending children of $x$, and $z_1, ..., z_d$ are the descending children. 
Finally, $\DFS(*)$ is defined similarly, but the root is not marked.

This is a walk of length $2\cdot \#(V\cup\{*\})$ on $T_*$. When convenient, we also consider it as indexed by the integers modulo $2n$; in that case we
refer to it as the \emph{cyclic DFS walk}.
\begin{ej}
    In the plane rooted tree of Figure \ref{fig:  route map}, the marked DFS walk on $T_*$ started at the root is $* \DFS(1) * \DFS(2) *$. Expanding the walk recursively,
    \[
    \DFS(*) = * ~ 1 ~ \circled{3} ~ 1 ~ 4 ~ \circled{8} ~ \circled{4} ~ \circled{7} ~ 4 ~ \circled{9} ~ 4 ~ \circled{1} ~ * ~ 2 ~ \circled{6} ~ \circled{2} ~ \circled{5} ~ 2 ~ *.
    \]
\end{ej}

It is shown in \cite{BEGLR} that the underlying permutation $\bar{\nu}_x \circ\iota_x\circ \bar{\sigma}_x$ of the unique non-intersecting family of network paths of $\Gamma_\mathcal{R}(x)$ is the cycle given by the marks of the marked DFS walk for every $x\in V$.
Moreover, if $\#\Asc(*) = 0$ and $\#\Des(*) = 2$ then this also holds for the root, as we will later see. In the example above, the permutation is $(3~8~4~7~9~1~6~2~5)$. For now, the precise statement that we need is the following, and it is shown by a careful inspection of Lemmas \ref{lem: underlying permutation south} and \ref{lem: underlying permutation north}.

\begin{lem}\label{lem: DFS}
    Let $x\in V$ with parent $X$.
    If the subsequence of marked nodes of the marked DFS walk $\DFS(x)$ is $w_1 ~ w_2 ~ \cdots w_k$, then
    $\bar{\nu}_x \circ \bar{\sigma}_x$ is the following:
    \[
    \scriptsize
    \begin{tikzpicture}[baseline = -1.5em, x = 3.5em, y = 3em, yscale = -1]
    \foreach \i in {1, ..., 5}{
    \node (\i) at (\i,0) {};
    }
    \foreach \i in {2,...,6}{
    \filldraw (\i) circle (1.5pt);
    }
    \filldraw (1)++(0,1) circle (1.5pt) node[below] {$e'(X,x)$\qquad\,} circle (1.5pt);
    \filldraw (2) node[above] {$\nabla_{w_1}$} ++(0,1) node[below] {$\Delta_{w_1}$} circle (1.5pt);
    \filldraw (3) node[above] {$\nabla_{w_2}$} ++(0,1) node[below] {$\Delta_{w_2}$} circle (1.5pt);
    \filldraw (4) node[above] {$\cdots$} ++(0,1) node[below] {$\cdots$} circle (1.5pt);
    \filldraw (5) node[above] {$\nabla_{w_k}$} ++(0,1) node[below] {$\Delta_{w_k}$} circle (1.5pt);
    \filldraw (6) node[above] {$e(x,X)$} circle (1.5pt);

    \begin{scope}[semithick, ->]
        \draw (1)++(0,1) -- (2);
        \draw (2)++(0,1) -- (3);
        \draw (3)++(0,1) -- (4);
        \draw (4)++(0,1) -- (5);
        \draw (5)++(0,1) -- (6);
    \end{scope}
    \end{tikzpicture}
\]
\end{lem}

\begin{lem}[Lemma 8.7 in \cite{BEGLR}]
\label{lem:interlacing}
In the cyclic DFS walk, marked nodes and arcs in $A_*$
are interlaced. That is, if
$\DFS(*) = w_1 ~ w_2 ~ \cdots w_k$ and $w_u$ and $w_v$ are two marked nodes, consecutive in $\DFS(*)$,
then there is a unique step in the subwalk 
$w_u~w_{u+1} \cdots w_v$ that is an arc of $A_*$.
Furthermore, every step in $w_u~w_{u+1} \cdots w_v$ before this one points from child to parent, and every step after this one points parent to child.
\end{lem}

\subsection{The local network of the root}
When $\#\Asc(*) = 0$ and $\#\Des(*) = 2$, then $\Gamma_\mathcal{R}(*)$ admits a unique non-intersecting path, \cite[Lemma 7.6]{BEGLR}. 
Recall that when talking about the network structure of the route map, one can identify the underlying map with a permutation in $\Sym_V$.
\begin{p}\label{prop: dfs da un ciclo}
    Let $T_*$ be a tree with a plane rooted tree structure induced by an arrowflow $A_*$. If $\#\Asc(*) = 0$ and $\#\Des(*) = 2$, then there exists a unique non-intersecting family of network paths in $\mathcal{R}$. Furthermore, its underlying permutation is the cycle of $\Sym_V$ produced by the markings of the marked DFS walk on $T_*$.
\end{p}
\begin{proof}
    This is Proposition 8.10 and Lemma 8.11 in \cite{BEGLR}.
\end{proof}

When the assumption $\#\Asc(*) = 0$ and $\#\Des(*) = 2$ is dropped, the analysis is slightly more complicated, but of remarkable elegance. A crucial step in the analysis is the first path-enumerative proof of the following result, given in \cite{derangements}. For another bijective proof, see \cite{Chapman}.

\begin{thm}
    \label{thm:enumeration of derangements}
    The signed enumeration of derangements is
    \[
    \sum_{\substack{\sigma \in \Sym_n\\ \text{derangement}}} \sgn(\sigma) = (-1)^{n-1}(n-1).
    \]
\end{thm}

The enumeration of derangements is reduced to a path enumeration problem on the following network.

\begin{de}
    Fix an integer $n\ge2$.
    Let $\mathcal{D}_n$ be a directed graph with
    \begin{itemize}
        \item two nodes $\Delta_i$ and $\nabla_i$ for each $1\le i\le n$, and
        \item two nodes $s(i,i+1)$ and $s(i+1,i)$ for each $1\le i < n$, and
    \end{itemize}
    We refer to the nodes as \emph{sources},  \emph{sinks}, and \emph{steps}, respectively.
    The arc-set has
    \begin{itemize}
        \item arcs $\big(\Delta_i, s(i,i+1)\big)$ and $\big(\Delta_{i+1}, s(i+1,i)\big)$ for each $1\le i < n$, 
        \item arcs $\big(s(i+1,i), \nabla_i\big)$ and $\big(s(i,i+1), \Delta_{i+1}\big)$ for each $1\le i < n$, and
        \item arcs $\big(s(i,i+1), s(i+1,i+2)\big)$ 
        and $\big(s(i+2,i+1), s(i+1,i)\big)$ 
        for each $1 \le i \le n-2$.
    \end{itemize}
    The \emph{derangement network of $\Sym_n$} is the network
    \(
    \big(\mathcal{D}_n, (\Delta_i)_{i\in[n]}, (\nabla_i)_{i\in[n]}\big).
    \)
    See Figure \ref{fig: derangement graph}.
\end{de}

\begin{figure}[h]
\[\tiny\begin{tikzcd}[cramped, column sep = tiny]
	& {s(1,2)} && {s(2,3)} && {s(3,4)} \\
	{\Delta_1} && {\Delta_2} && {\Delta_3} && {\Delta_4} \\
	{\color{gray}\nabla_1} && {\color{gray}\nabla_2} && {\color{gray}\nabla_3} && {\color{gray}\nabla_4} \\
	& {s(2,1)} && {s(3,2)} && {s(4,3)}
	\arrow[from=1-2, to=1-4]
	\arrow[bend right, gray, from=1-2, to=3-3]
	\arrow[from=1-4, to=1-6]
	\arrow[bend right, gray, from=1-4, to=3-5]
	\arrow[bend right, gray, from=1-6, to=3-7]
	\arrow[from=2-1, to=1-2]
	\arrow[from=2-3, to=1-4]
	\arrow[bend right, crossing over, from=2-3, to=4-2]
	\arrow[from=2-5, to=1-6]
	\arrow[bend right, crossing over, from=2-5, to=4-4]
	\arrow[bend right, crossing over, from=2-7, to=4-6]
	\arrow[gray, from=4-2, to=3-1]
	\arrow[gray, from=4-4, to=3-3]
	\arrow[from=4-4, to=4-2]
	\arrow[gray, from=4-6, to=3-5]
	\arrow[from=4-6, to=4-4]
\end{tikzcd}\]
    \caption{The derangement network of $\Sym_4$.}
    \label{fig: derangement graph}
\end{figure}

In plain words, the derangement network has two paths, one recording ``increasing steps'' and one recording ``decreasing steps''. A network path starts in a source corresponding to a number $i \in [n]$ and decides whether to increase or decrease. It follows one of the two paths accordingly and exits at a sink, corresponding to another number $j\in[n]$. In a family of network paths, ``increasing steps'' are visited by the path corresponding to \emph{excedances} $i < \sigma(i)$ of the underlying permutation $\sigma$ of the family, and ``decreasing steps'' by the \emph{anti-excedances}.

\begin{lem}
\label{lem:derangement network}
    For each derangement $\sigma\in\Sym_n$, there is a unique family of network paths $\Lambda_\sigma$ in $\mathcal{D}_n$ with underlying permutation $\sigma$. This is furthermore a bijection between $\FNP(\mathcal{D}_n)$ and the set of derangements of $\Sym_n$.
\end{lem}
\begin{proof}
See \cite{derangements}.
\end{proof}

It turns out that derangement networks play an important role in our study of the minors of the distance matrix of $T$. The relationship is as follows: 

\begin{p} \label{prop:derangements1}
    Fix a composite $m$-arrowflow $A$. Let $\mathcal{R}$ be the corresponding route map of $A/\sim$.
    Let $a = \#\Asc(*)$, let $d = \#\Des(*)$, and consider the local network $\Gamma_{\mathcal{R}}(*)$.
    There is a graph endomorphism from $\mathcal{D}_a\times\mathcal{D}_d$ to $\Gamma_\mathcal{R}(*)$ inducing a bijection between the sets of non-intersecting families of network paths.
\end{p}

\begin{proof}
  The first thing to notice is that every network path of the family has to start at the Southern local network and end at the Northern local network, so it has to travel through a bridge between hemispheres. 
    There are as many bridges as paths, so in order for a family to not be intersecting, each path should take a different bridge.
    If $j\in\Asc(*)$ is an ascending child of the root, then $e(j,*)$ is both a source and the start of a bridge; in a non-intersecting family of paths, the path $\Lambda_j$ starting at $e(j,*)$ must also be the unique path taking the bridge $\big(e(j,*),e'(j,*)\big)$.
    Note that when $\#\Asc(*) = 1$, this path cannot end in any sink; there exists no non-intersecting family in this case. Similarly, if $j\in\Des(*)$, then the path ending at $e'(*,j)$ must also be the unique path taking the bridge $\big(e(*,j),e'(*,j)\big)$; and if $\#\Des(*)=1$ there exists no non-intersecting family.

    Altogether, we have shown that for a family $\Lambda\in\FNP(\Gamma_{\mathcal{R}}(*))$ to be non-intersecting, it has to split into two non-intersecting families of network paths in the networks
    \begin{align*}
     \Gamma^{\Asc}(*) &:= \Big(
    \Gamma_{\mathcal{N}}(*), \big(e'(j,*)\big)_{j\in\Asc(*)}, \big(e'(*,j)\big)_{j\in\Asc(*)}
    \Big),
    \text{ and}\\
    \Gamma^{\Des}(*) &:= \Big(\Gamma_{\mathcal{S}}(*), \big(e(j,*)\big)_{j\in\Des(*)}, \big(e(*,j)\big)_{j\in\Des(*)}
    \Big).   
    \end{align*}
    Finally, consider the map
    \begin{align*}
        \mathcal{D}_{\Asc(*)} &\to \Gamma^{\Asc}(*)\\
        \Delta_i &\mapsto e'(i,*)\\
        \nabla_i &\mapsto e'(*,i)\\
        s(i,i+1) &\mapsto s_*(i, i+1).
    \end{align*}
    This is an endomorphism of digraphs when $\#\Asc(*)\ge2$, and by the above, the map induces a bijection between the sets of non-intersecting paths of the networks.
    Similarly, there is an endomorphism $\mathcal{D}_{\Des(*)} \to \Gamma^{\Des}(*)$  whenever $\#\Des(*)\ge2$ and the analogous property holds.
\end{proof}

\section{Unital arrowflows}
\label{sec:unital}

This section is dedicated to the proof of Proposition \ref{prop:unital}. 
Let $A$ be a unital arrowflow. We have 
\[
\edge{\SG(A)} = (n-1) - \edge{\missing{A}} = \cc{\missing{A}} - 1 = m-1
\]
and $\edge{A} = m$. Therefore, there is exactly one edge $\{u,v\}$ of $\SG(A)$ in which there are two arcs of $A$ are supported. Since $A$ has no parallel arrows, these two arcs must be $(u,v)$ and $(v,u)$; we call them \emph{anti-parallel arrows}.

\begin{lem}
    Fix $F\in\mathcal{F}_1(T;S)$ an $S$-rooted forest of $T$.
    We have 
    \[
    \#\{A~\text{unital, }\, \missing{A} = F\} = (m-1)2^{m-2}.
    \]
\end{lem}
\begin{proof}
    We have $\SG(A) = F^\comp$ and thus $\edge{\SG(A)} = m-1$.
    The factor $(m-1)$ counts the number of ways of choosing the edge of $\SG(A)$ in which the anti-parallel arrows are supported. For each of the $m-2$ remaining edges, there are two possible orientations.
\end{proof}

Recall that the quotient arrowflow $A/\sim$ is a \emph{complete} unital $m$-arrowflow on $T/\sim ~= (S, E/\sim)$. Let $\{u, v\}$ be the edge of $T/\sim$ in which the anti-parallel arrows are supported.
We construct a rooted tree $(T/\sim)_*$ from $T/\sim$ by subdividing $\{u, v\}$, and calling the newly introduced vertex $*$. Similarly, we construct a complete $m$-arrowflow $(A/\sim)_*$ on $(T/\sim)_*$ with arc-set
\[
\{(i,j) \in A/\sim \ : \ \{i,j\}\ne\{u, v\}\} \cup \{(*,u), (*, v)\}.
\]

\begin{lem}
\label{lem:sum of C_S(A) unital}
 If $A$ is unital, then   
\[  \sum_{\kappa \in C_S(A)} \sgn(\kappa) = (-1)^{m - 1}.
\]
\end{lem}
\begin{proof}
By Lemma \ref{lem:C_S(A) sum equals C_S(A/sim)}, we have 
\[
\sum_{\kappa\in C_S(A)}\sgn(\kappa) = \sum_{\kappa\in C_S(A/\sim)} \sgn(\kappa)
\]

The quotient $A/\sim$ is a \emph{complete} unital $m$-arrowflow on $T/\sim$. We follow \cite{BEGLR} to conclude:
let $\mathcal{R}$ be the route map of the rooted directed tree $(A/\sim)_*$. By Proposition \ref{prop:lift paths to routemap}\eqref{propItem:bijection catalysts routemap} and the LGV lemma, we have
\[
\sum_{\kappa \in C_S(A/\sim)} \sgn(\kappa) = \sum_{\substack{\Lambda\in\FNP(\mathcal{R})\\\text{full}}} \sgn(\Lambda) = \sum_{\substack{\Lambda\in\FNP(\mathcal{R})\\\text{non-intersecting}}} \sgn(\Lambda).
\]
By Proposition \ref{prop: dfs da un ciclo}, there exists a unique  non-intersecting family $\bar\Lambda$ of network paths of $\mathcal{R}$, and its sign is $(-1)^{m-1}$.
We conclude
\[
\sum_{\kappa\in C_S(A)}\sgn(\kappa) =  \sum_{\substack{\Lambda\in\FNP(\mathcal{R})\\\text{non-intersecting}}} \sgn(\Lambda) = \sgn(\bar\Lambda) = (-1)^{m-1}. \qedhere
\]
\end{proof}

We emphasize that the arguments used in \cite{BEGLR} are involutive. In particular, our proof of Lemma \ref{lem:sum of C_S(A) unital} is also involutive.

\section{Composite arrowflows}
\label{sec:composite}

Fix $F$ an $(S,*)$-rooted spanning forest of $T$.
This section is devoted to the proof of Proposition \ref{prop:composite}, namely,
\[
\sum_{\substack{A\\\text{composite}\\ \missing{A} = F}}\sum_{\kappa\in C_S(A)} \sgn(\kappa) =(-1)^{m-1} 2^{m - 2}(\bdeg{F_*} - 1)(\bdeg{F_*} - 4).
\]
Recall that the vertex set $I$ of $T/\sim$ is $S\sqcup\{*\}$, where $*$ denotes the quotient projection of the floating component of $\missing{A}$.
We endow $T/\sim$ with a structure of rooted tree by letting $*$ be the root.
We then note that $A/\sim$ is a rooted directed tree supported on $T/\sim$.
Endow $T/\sim$ with a plane directed tree structure induced by $A/\sim$.
We let $\mathcal{R}$ denote the \emph{route map} of $A/\sim$.

We will shortly compute a formula for the signed enumeration of the catalysts in a composite arrowflow class.
This proof will follow the structure of the analogous result for unital arrowflows.

\begin{p}
\label{prop:derangements2}
    Fix a composite $m$-arrowflow $A$. Let $\mathcal{R}$ be the corresponding route map of $A/\sim$.
    Let $a = \#\Asc(*)$, let $d = \#\Des(*)$, and consider the local network $\Gamma_{\mathcal{R}}(*)$.
    There exists exactly one family of network paths of $\Gamma_{\mathcal{R}}(*)$
    with underlying permutation $\tau$ for each 
    derangement $\tau \in \Sym_d\times\Sym_a$. Every other family is intersecting.
    Consequently, in absolute value, we have
    \[
    \left |\sum_{\Lambda\in\FNP(\Gamma_{\mathcal{R}}(*))} \sgn(\Lambda) \right | = \begin{cases}
        a-1 & \text{if }d=0,\\
        d-1 & \text{if }a=0,\\
        (a-1)(d-1) & \text{otherwise.}
    \end{cases}
    \]
\end{p}
\begin{proof}
    By an application of the LGV lemma and Proposition \ref{prop:derangements1}, we have
    \begin{align*}
        \sum_{\Lambda\in\FNP(\Gamma_{\mathcal{R}}(*))}  \sgn(\Lambda) &=
    \sum_{\substack{\Lambda\in\FNP(\Gamma_{\mathcal{R}}(*))\\\text{non-intersecting}}}\sgn(\Lambda) \\
    & =
    \bigg(\sum_{\substack{\Lambda\in\FNP(\mathcal{D}_a)\\\text{non-intersecting}}} \sgn(\Lambda)\bigg)\bigg( \sum_{\substack{\Lambda\in\FNP(\mathcal{D}_d)\\\text{non-intersecting}}} \sgn(\Lambda)\bigg)
    \end{align*}
    Lemma \ref{lem:derangement network} and Theorem \ref{thm:enumeration of derangements} then give the desired formula.
\end{proof}

\begin{cor}
    Fix a composite $m$-arrowflow $A$ and let $\mathcal{R}$ be the corresponding route map of $A/\sim$. 
    Let $a = \#\Asc(*)$, $d = \#\Des(*)$, and
    \[
    (\partial x, <_x) = (y_1 <_x \cdots <_x y_{a+d} <_x X).
    \]
    There exists exactly one 
    non-intersecting family of network paths in $\mathcal{R}$
    with underlying permutation 
    \[
    \overline{\nu}_{y_{a+d}} \circ\cdots\circ \overline{\nu}_{y_1} \circ \tau \circ \overline{\sigma}_{y_1} \circ\cdots\circ \overline{\sigma}_{y_{a+d}}
    \]
    for each
    derangement $\tau \in \Sym_d\times\Sym_a$,
    where each other factor is unique and defined as in Lemma \ref{lem:underlying map}. This is furthermore an exhaustive enumeration of non-intersecting families of $\FNP(\mathcal{R})$.
\end{cor}
\begin{proof}
    Existence is guaranteed by Lemma \ref{lem:derangement network}, Proposition \ref{prop:derangements1} and Lemma \ref{lem:underlying map}, which also imply the list is exhaustive.
\end{proof}

\begin{cor}\label{cor:composite NIP}
    Fix a composite $m$-arrowflow $A$ and let $\mathcal{R}$ be the corresponding route map of $A/\sim$. 
    Let $a = \#\Asc(*)$, $d = \#\Des(*)$. Then,
    \[
    \sum_{\substack{\Lambda\in\FNP(\mathcal{R})\\\text{non-intersecting}}} \sgn(\Lambda) = (-1)^{m-1}(a-1)(d-1).
    \]
\end{cor}
\begin{proof}
    Consider
    \begin{align*}
        \sum_{\substack{\Lambda\in\FNP(\mathcal{R})\\\text{non-intersecting}}} \sgn(\Lambda) &=
        \sum_{\substack{\Lambda\in\FNP(\mathcal{R})\\\text{non-intersecting}}} \sgn(\overline{\nu}_{y_{a+d}} \circ\cdots\circ \overline{\nu}_{y_1} \circ \rho_* \circ \overline{\sigma}_{y_1} \circ\cdots\circ \overline{\sigma}_{y_{a+d}})\\
        &=
        \sum_{\substack{\Lambda\in\FNP(\Gamma_\mathcal{R}(*))\\\text{non-intersecting}}} \sgn(\overline{\nu}_{y_{a+d}} \circ\cdots\circ \overline{\nu}_{y_1} \circ \tau \circ \overline{\sigma}_{y_1} \circ\cdots\circ \overline{\sigma}_{y_{a+d}}),
    \end{align*}
    where $\tau$ is the derangement corresponding to $\Lambda$ after Lemma \ref{lem:derangement network} and Proposition \ref{prop:derangements1}, and the remaining maps $\overline{\nu}_{y_{a+d}}, ..., \overline{\sigma}_{y_{a+d}}$ are independent of said $\Lambda$ by Lemma \ref{lem:underlying map}.
    The whole composite is a permutation of $S$, and $\tau$ is a permutation of $\partial*$. Moreover, $\tau$ is a permutation of $\Sym_{a}\times\Sym_{d}$; we write $\tau = \tau_{\Asc}\tau_{\Des}$.

    Let $y_1, ..., y_a$ be the ascending children of $*$ and denote by $z_1, ..., z_d$ the descending children.
    Let $X = \bigcup_{y\in\Asc(*)} V(T_y)$ be the set formed by the ascending children of the root, and their children, grandchildren, etc. Let $Y = \bigcup_{z\in\Des(*)} V(T_z)$ be defined similarly. Then, we can find $\alpha_1, \alpha_2 \in \Sym_{X}$ and $\beta_1, \beta_2 \in\Sym_Y$ such that the composite above can be written as a product $\alpha_2\tau_{\Asc}'\alpha_1 ~ \beta_2\tau_{\Des}'\beta_1$ of permutations, where $\tau'_{\Asc}$ is the permutation of $\Sym_S$ that agrees with $\alpha_2\tau_{\Asc}\alpha_1^{-1}$ in $\Asc(*)$ and is the identity elsewhere, and $\tau_{\Des}'$ is defined similarly.
    \[
    \begin{tikzpicture}[x = 1em, y = -1.5em, baseline = -6em]
        \foreach \i in {0,1,2,3}{
            \filldraw (\i,7) circle (1.5pt);
            \filldraw (\i,5) circle (1.5pt);
            \draw[->, shorten >= 2pt, semithick] (\i,7) to (\i,5);
        }
        \foreach \i in {2,3}{
            \filldraw (\i,5) circle (1.5pt);
            \filldraw (\i,3) circle (1.5pt);
            \draw[->, shorten >= 2pt, semithick] (\i,5) to (\i,3);
        }
        \foreach \i in {2,3,4,5}{
            \filldraw (\i,3) circle (1.5pt);
            \filldraw (\i,1) circle (1.5pt);
            \draw[->, shorten >= 2pt, semithick] (\i,3) to (\i,1);
        }
        \node[rectangle, draw, fill = white] (*) at (1.5,6.1) {\scriptsize$\overline{\sigma}_{y_1} \circ\cdots\circ \overline{\sigma}_{y_{a+d}}$};
        \node[rectangle, draw, fill = white] (*) at (2.5,4.1) {$~~\tau~~$};
        \node[rectangle, draw, fill = white] (*) at (3.5,2.1) {\scriptsize$\overline{\nu}_{y_{a+d}} \circ\cdots\circ \overline{\nu}_{y_1}$};
    \end{tikzpicture}
    \quad=\quad
    \begin{tikzpicture}[x = 1em, y = -1.5em, baseline = -6em]
        \foreach \i in {0,1,2,3,4,5}{
            \filldraw (\i,7) circle (1.5pt);
            \filldraw (\i,1) circle (1.5pt);
            \draw[->, shorten >= 2pt, semithick] (\i,7) to (\i,1);
        }
        \node[rectangle, draw, fill = white] (*) at (1,5.5) {\scriptsize$~\phantom{\beta}\alpha_2~~~$};
        \node[rectangle, draw, fill = white] (*) at (1,4) {\scriptsize$~~\tau_{\Asc}'~$};
        \node[rectangle, draw, fill = white] (*) at (1,2.5) {\scriptsize$~\phantom{\beta}\alpha_1~~~$};
        \node[rectangle, draw, fill = white] (*) at (4,5.5) {\scriptsize$~~~\beta_2~~~$};
        \node[rectangle, draw, fill = white] (*) at (4,4) {\scriptsize$~~\tau_{\Des}'~$};
        \node[rectangle, draw, fill = white] (*) at (4,2.5) {\scriptsize$~~~\beta_1~~~$};
    \end{tikzpicture}
    \]
    Then, $\sgn(\alpha_2\tau_{\Asc}'\alpha_1 ~ \beta_2\tau_{\Des}'\beta_1) = \sgn(\alpha_2\alpha_1 ~ \beta_2\beta_1)\sgn(\tau)$ and the sum simplifies to
    \begin{align*}
        \sum_{\substack{\Lambda\in\FNP(\mathcal{R})\\\text{non-intersecting}}} \sgn(\Lambda) &=
        \sum_{\substack{\Lambda\in\FNP(\Gamma_\mathcal{R}(*))\\\text{non-intersecting}}} \sgn(\alpha_2\tau_{\Asc}'\alpha_1 ~ \beta_2\tau_{\Des}'\beta_1)\\
        &=
        \sgn(\alpha_2\alpha_1 ~ \beta_2\beta_1) \sum_{\substack{\Lambda\in\FNP(\Gamma_\mathcal{R}(*))\\\text{non-intersecting}}} \sgn(\tau)\\
        &=
        \sgn(\alpha_2\alpha_1 ~ \beta_2\beta_1)(-1)^{\delta_{a\ne0}}(-1)^{\delta_{d\ne0}} (a-1)(d-1).
    \end{align*}
    To compute $\sgn(\alpha_2\alpha_1 ~ \beta_2\beta_1)(-1)^{\delta_{a\ne0}}(-1)^{\delta_{d\ne0}}$, we argue as follows. We first assume $a > 0$.
    Let $\tau'_{\Asc}$ be the cycle $(y_1~y_2~\cdots~y_a)$. For each $y\in\Asc(*)$, by Lemma \ref{lem: DFS}, the underlying map of the unique non-intersecting family of network paths of $\Sigma_\mathcal{S}(y)\doubleplus\mathcal{B}(y)\doubleplus\Sigma_\mathcal{N}(y)$ is given by the DFS walk on $T_*$ and of this form:
    \[
    \scriptsize
    \begin{tikzpicture}[baseline = -1.5em, x = 2em, y = 3em, yscale = -1]
    \foreach \i in {1, ..., 5}{
    \node (\i) at (\i,0) {};
    }
    \foreach \i in {2,3,4}{
    \filldraw (\i) circle (1.5pt);
    }
    \filldraw (1)++(0,1) circle (1.5pt) node[below] {$e'(*,y)$\qquad\,} circle (1.5pt);
    \filldraw (2) node[above] {} ++(0,1) node[below] {} circle (1.5pt);
    \filldraw (3) node[above] {} ++(0,1) node[below] {} circle (1.5pt);
    \filldraw (4) node[above] {} ++(0,1) node[below] {} circle (1.5pt);
    \filldraw (5) circle (1.5pt) node[above] {\qquad$e(y,*)$};

    \begin{scope}[semithick, ->]
        \draw (1)++(0,1) -- (2);
        \draw (2)++(0,1) -- (3);
        \draw (3)++(0,1) -- (4);
        \draw (4)++(0,1) -- (5);
    \end{scope}
    \end{tikzpicture}
\]
By letting $\tau_{\Asc}'$ be the cycle $(y_1~y_2~\cdots~y_a)$, the marked Depth-First-Search walk around $\bigcup_{y\in\Asc(*)} T_y$ gives the map $\alpha_2\tau_{\Asc}'\alpha_1$. But the Depth-First-Search always gives a cycle. Hence,
\[
\sgn(\alpha_2\tau_{\Asc}'\alpha_1) = (-1)^{\#X-1} = \sgn(\alpha_2\alpha_1)\sgn(\tau_{\Asc}')
\]
and $\sgn(\tau_{\Asc}') = (-1)^{a-1}$. Similarly, if $d > 0$, then
\[
\sgn(\beta_2\tau_{\Des}'\beta_1) = (-1)^{\#Y-1} = \sgn(\beta_2\beta_1)\sgn(\tau_{\Des}')
\]
and $\sgn(\tau_{\Des}') = (-1)^{d-1}$. Altogether, if $a>0$ and $d>0$, we have
\[
\sgn(\alpha_2\tau_{\Asc}'\alpha_1 ~ \beta_2\tau_{\Des}'\beta_1) =
(-1)^{\#X + \#Y} = (-1)^m =
\sgn(\alpha_2\alpha_1 ~ \beta_2\beta_1)\sgn(\tau')
\]
and $\sgn(\tau') = (-1)^{a+d} = (-1)^{\delta(*)}$.
And therefore $\sgn(\alpha_2\alpha_1 ~ \beta_2\beta_1) = (-1)^m$.

If $a = 0$, then $\sgn(\alpha_2\alpha_1 ~ \beta_2\beta_1) = \sgn(\beta_2\beta_1) = (-1)^{m-1}$. Similarly if $d=0$. In any case, 
\[
\sgn(\alpha_2\alpha_1 ~ \beta_2\beta_1)(-1)^{\delta_{a\ne0}}(-1)^{\delta_{d\ne0}} = (-1)^m. \qedhere
\]
\end{proof}
In order for our proof of Theorem \ref{thm:main result} to be completely involutive, we give a bijective proof of the following binomial identity.
\begin{lem}\label{lem: el puto 4}
For $n \geq 0$,
    \[
    \sum_{k = 0}^{n}\binom{n}{k}(k-1)(n-k-1) = 2^{n-2}n(n-1) - 2^n (n-1).
    \] 
\end{lem}
\begin{proof}
    We begin by noting that the summands corresponding to $k = 1$ and $k = n-1$ vanish, while the summands corresponding to $k = 0$ and $k=n$ are negative and contribute to the sum with
    \[
    \binom{n}{0}(-1)(n-1) + \binom{n}{n}(n-1)(-1) = -2(n-1).
    \]
    We are looking for a bijective proof, so this term is passed to the other side of the equality as preparation. We want to show
    \begin{equation}\label{eq:conteo}
    \sum_{k = 2}^{n-2}\binom{n}{k}(k-1)(n-k-1) = 2^{n-2}n(n-1) - \Big( 2^n (n-1) - 2(n-1) \Big).
    \end{equation}
    Consider the set $R$ of sequences of black and white beads and of length $n$, in which exactly one black bead and one white bead are distinguished.
    For instance, for $n = 9$ then
    \(\owhite \oblack \owhite \oblack \hat{\oblack} \owhite \owhite \hat{\owhite} \oblack\)
    is in $R$, and also \(\oblack \oblack  \hat{\owhite} \owhite  \oblack  \owhite  \oblack \hat{\oblack} \owhite\). The cardinality of this set can be computed as follows: start with a sequence of $n$ beads and choose two entries of the sequence to be distinguished, which you can do in $\binom{n}{2}$ ways; then color all of the beads either black or white except the left-most of the distinguished beads, which you can do in $2^{n-1}$ ways; finally, color the remaining bead such that both distinguished beads have different colors. Altogether,
    \[
    \#R = 2^{n-1}\binom{n}{2} = 2^{n-2}n(n-1).
    \]
    Now, consider the subset $L$ of $R$ for which the distinguished black bead is \emph{not} the left-most black bead, and the distinguished white bead is \emph{not} the left-most white bead. The first of the two examples above is in $L$, whereas the second is not, since $\hat{\owhite}$ is the left-most white bead. The set $L$ is counted as follows: start with a sequence of $n$ white beads, and choose a number $k \in \{2, ..., n-2\}$; then select $\binom{n}{k}$ entries that will be colored black. Subsequently, choose a black bead to be distinguished (but not the left-most one), which you can do in $(k-1)$ ways; finally, select a white bead to be distinguished (but not the left-most one), which you can do in $(n-k-1)$ ways. Hence
    \[
    \# L = \sum_{k = 2}^{n-2}\binom{n}{k}(k-1)(n-k-1).
    \]
    To conclude, we compute the cardinality of $R\setminus L$, which reinterprets Equation \eqref{eq:conteo} as $\# L = \#R - \#(R\setminus L)$. A sequence is in $R\setminus L$ if $\hat{\oblack}$ is the left-most black bead, or if $\hat{\owhite}$ is the left-most white bead, or both. This set is counted as follows: start by coloring a sequence of $n$ beads, which you can do in $2^n$ ways.
    Discard the coloring if all beads are black or all are white ---this sequence can't be in $R$---; there are $2^n-2$ possible colorings.
    You will now have $k$ black and $n-k$ white beads, for some $k$. If $0<k<n$, distinguish the first black bead and any white bead, which you can do in $n-k$ ways, or the first white bead and any black bead except the left-most one, which you can do in $k-1$ ways. Hence there are always $n-1$ ways of distinguishing two beads. Altogether,
    \[
    \#(R\setminus L) = (2^{n} - 2)(n-1),
    \]
    and Equation \eqref{eq:conteo} becomes $\# L = \#R - \#(R\setminus L)$.
\end{proof}

\begin{proof}[Proof of Proposition \ref{prop:composite}]
Fix a forest $F\in\SSrooted$. Proposition \ref{prop:lift paths to routemap} gives
\[
\sum_{\substack{A\\\text{composite}\\ \missing{A} = F}}\sum_{\kappa\in C_S(A)} \sgn(\kappa) =
\sum_{\substack{A\\\text{composite}\\ \missing{A} = F}}\sum_{\substack{\Lambda\in\FNP(\mathcal{R})\\\text{non-intersecting}}} \sgn(\Lambda),
\]
where the route map $\mathcal{R}$ of the second sum is constructed with respect to the arrowflow $A$. Then, Corollary \ref{cor:composite NIP} gives
\[
\sum_{\substack{A\\\text{composite}\\ \missing{A} = F}}\sum_{\substack{\Lambda\in\FNP(\mathcal{R})\\\text{non-intersecting}}} \sgn(\Lambda)
=
\sum_{\substack{A\\\text{composite}\\ \missing{A} = F}} (-1)^{m}(\#\Asc(*)-1)(\#\Des(*)-1),
\]
where again $\Asc(*)$ and $\Des(*)$ are taken with respect to the arrowflow $A$. 
Let $F_*$ be the floating component of $F$. If $A$ is a composite arrowflow with missing forest $F$, then $F_*$ becomes $*$ in the quotient $A/\sim$. Thus a composite arrowflow 
with missing forest $F$
is constructed by choosing $a$ vertices among $\partial F_*$ to become ascending children of the root, which also determines the remaining $\delta(F_*) - a$ as descending children of the root; the other $m-\delta(F_*)$ edges of $F^\comp$ can be oriented arbitrarily. Hence, the sum above becomes
\[
(-1)^{m}\sum_{a = 0}^{\delta(F_*)} \binom{\delta(F_*)}{a} (a-1)(\delta(F_*)-a-1) \cdot 2^{m-\delta(F_*)}.
\]
Now, Lemma \ref{lem: el puto 4} gives
\[
(-1)^{m}\cdot 2^{\delta(F_*)-2}(\delta(F_*)-1)(\delta(F_*)-4)
\cdot 2^{m-\delta(F_*)}. \qedhere
\]
\end{proof}

\section{
A theorem of Richman, Shokrieh, and Wu}
\label{sec: richman as coro}
In this section, we derive \cite[Thm.~1.1]{richmanminors} from Theorem \ref{thm:main result}. We begin by stating the theorem.

\begin{thm}[\cite{richmanminors}]
      \label{eqn:richman main result}
      \label{thm:richman main result}
    Let $T = ([n], E)$ be a tree, and let $S\subseteq[n]$ be a subset of cardinality $m\ge2$. Then, 
    \[
      \det D[S] = (-1)^{m - 1} 2^{m - 2}
      \left(
      (n-1)\kappaTS - \sum_{F \in \SSrooted} \big(\bdeg{F_*} - 2\big)^2
      \right).
    \]
\end{thm}

We give the first combinatorial proof. Before, we give two preliminary results. The first one is \cite[Rk.~2.5]{richmanminors}.

\begin{lem}
    Fix an edge $e \in T$ and a spanning forest $F$ of $T$.
    \label{lem: remove edge from Srooted}
    \begin{enumerate}[label = (\alph*)]
        \item If $F$ is $S$-rooted and $e \in F$, then $F \setminus \{e\}$ is $(S, *)$-rooted and $e$ is adjacent to the floating component of $F \setminus \{e\}$.
        \item If $F$ is $(S, *)$-rooted and $e$ is adjacent to $F_*$, then $F \cup \{e\}$ is $S$-rooted.
    \end{enumerate}
\end{lem}
\begin{proof}
    Removing an edge from a tree results in two connected components; removing an edge $e$ from an $S$-rooted forest results in a forest with $m+1$ connected components. Furthermore, each component has a node of $S$ except one of the two that was adjacent to $e$. This latter becomes the floating component of the $(S,*)$-rooted forest $F\setminus \{e\}$. Note that $e$ is adjacent to the floating component, as claimed. This shows part (a).
    Part (b) is immediate from the definitions.
\end{proof}

\begin{lem} Let $T = ([n], E)$ be a tree, let $S$ be an $m$-subset of $[n]$. Then,
\label{lem: sum of bdeg}
    \[(n - m) \cdot \#\Srooted = \sum_{F \in \SSrooted} \bdeg{F_*}.\]
\end{lem}
\begin{proof}
Consider the sets 
\begin{align*}
A &:= \{(F, e) : \  F \in \Srooted, e \in F\}\subseteq \Srooted \times E, \\
B &:= \{(F, e) : \ F \in \SSrooted, e \text{ adjacent to } F_*\}\subseteq \SSrooted \times E.
\end{align*}
Note that $\#A = (n - m)\#\Srooted$ and $\#B = \sum_{F \in \SSrooted} \bdeg{F_*}$. It is thus sufficient to exhibit a bijection between $A$ and $B$.
Consider the map $h : A \to B$  defined by
\[h(F, e) = (F \setminus \{e\}, e), \quad (F, e) \in A.\]
This map is bijective: the inverse of $h$ exists and is given by
\[
h^{-1}(F, e) = (F \cup \{e\}, e), \quad (F, e) \in B.
\]
Note that both $h$ and $h^{-1}$ are well-defined by Lemma \ref{lem: remove edge from Srooted}.
\end{proof}  
\begin{proof}[Proof of Theorem \ref{thm:richman main result}]
By Theorem \ref{thm:main result},
\[
\frac{\det D[S]}{(-1)^{m-1} 2^{m-2}} = (m - 1)\kappaTS - \!\!\!\!\sum_{\SSrooted} \!\!\!\!\big(\bdeg{F_*} - 1\big)\big(\bdeg{F_*} - 4\big).
\]
Lemma \ref{lem: sum of bdeg} then yields
\begin{align*}
\frac{\det D[S]}{(-1)^{m-1} 2^{m-2}}&=
\frac{\det D[S]}{(-1)^{m-1} 2^{m-2}} + (n - m)\Srooted - \sum_{\SSrooted}\bdeg{F_*}\\
&=(m - 1 + n - m)\kappaTS - \!\!\!\!\sum_{\SSrooted} \!\!\!\!\big(\bdeg{F_*} +\big(\bdeg{F_*} - 1\big)\big(\bdeg{F_*} - 4\big)\big)\\
&=(n - 1)\kappaTS - \!\!\!\!\sum_{\SSrooted} \!\!\!\!\big( \bdeg{F_*} - 2\big)^2,
\end{align*}
as desired. 
\end{proof}

\section{
A corollary of Choudhury and Khare}
\label{sec: CK as coro}

Although stated in greater generality (by considering weighted graphs, and computing the whole characteristic polynomial of the matrix), an immediate corollary of 
\cite[Thm.~A]{CK19} is the following.
\begin{cor}[\cite{CK19}]
    Let $S$ be a subset of nodes of $T$ of cardinality $m \ge 3$. Assume that the induced subgraph of $T$ on $S$ is a tree. Then,
    \[
    \det D[S] = (-1)^{m-1}\cdot  (m-1) \cdot 2^{m-2}. 
    \]
\end{cor}
The hypothesis on $S$ can be rephrased as saying that $T$ is constructed from a tree $S$ by iteratively attaching pendant trees to some of the nodes. We give the first combinatorial account of this result.
\begin{proof}
    We take the second, more intuitive notion: identify $S$ with a tree and think of $T$ as a tree constructed from $S$ by iteratively attaching pendant trees to some of the nodes.
    There is only one $S$-rooted forest on $T$, namely the forest in which each components is an elements $s \in S$ together with all of the nodes of all of the pendant trees attached to $s$.
    In an $(S,*)$-rooted forest, the floating component must be one of the pendant trees; in particular, $\delta (F_*) = 1$ for all pendant trees. Now, the formula from Theorem \ref{thm:main result} gives
    \[
    \det D[S] =
    (-1)^{m-1} 2^{m-2}
    \Big(
    (m-1)\cdot 1 - \sum_{F \in \SSrooted} 0
    \Big). \qedhere
    \]
\end{proof}

\section{Closing remarks}
\label{sec: closing}

The present work perfectly illustrates how the combinatorial framework developed in \cite{BEGLR} is, at the same time, natural and powerful. Indeed, the setting of \cite{CK19} is the correct algebraic setting to study distance matrices of trees in all generality, and the results therein prove to be the shadow of beautiful network combinatorics.

We remark that this document can be further generalized in two directions. The first, towards computing non-principal minors of the matrix; the second towards computing weighted generalizations of the formula. Both directions require simple changes in the combinatorial model, followed by careful (and tedious) algebraic computations. We leave this for a future work, and choose to tell a simpler story here.

\section*{Acknowledgements}
The current version of this article was written to be presented at the British Combinatorial Conference 2024.
We thank Mercedes Rosas, Luis Esquivias, and Emmanuel Briand for helpful commentaries. We thank Harry Richman for making us aware of their preprint \cite{richmanminors} after the announcement of \cite{FPSAC}.

ÁG was funded by the
University of Bristol Research Training Support Grant. AL was partially supported by the Grant PID2020-117843GB-I00 funded by\break \mbox{MICIU/AEI/10.13039/501100011033}.

\end{document}